\documentclass[11pt]{amsart}
\usepackage{amsmath,amssymb,amsfonts,amscd,epsfig}

\usepackage{axodraw4j}
\usepackage{color}

\theoremstyle{plain}
\newtheorem{thm}{Theorem}
\newtheorem{lem}[thm]{Lemma}
\newtheorem{con}[thm]{Conjecture}
\newtheorem{cor}[thm]{Corollary}
\newtheorem{prop}[thm]{Proposition}
\newtheorem{remark}[thm]{Remark}
\newtheorem{defn}[thm]{Definition}
\newtheorem{propdef}[thm]{Proposition-Definition}
\newtheorem{ex}[thm]{Example}

\newcommand{\rank}{{\rm rank\,}}
\newcommand{\corank}{{\rm corank\,}}

\newcommand{\Sing}{{\rm Sing}}

\newcommand{\sQ}{{\mathcal Q}}

\newcommand{\A}{{\mathbb A}}

\newcommand{\FF}{{\mathbb F}}

\newcommand{\PP}{{\mathbb P}}
\newcommand{\QQ}{{\mathbb Q}}

\newcommand{\ZZ}{{\mathbb Z}}

\newcommand{\Lef}{\mathbb{L}}
\newcommand{\q}{/\!\!/}
\title[Properties of $c_2$ invariants of Feynman graphs]{Properties of $c_2$ invariants of Feynman graphs}
\author{Francis Brown, Oliver Schnetz, Karen Yeats}

\thanks{Karen Yeats is supported by an NSERC discovery grant and would like to thank Samson Black for explaining knots. Francis Brown is partially supported by ERC grant 257638.
Francis Brown and Oliver Schnetz thank Humboldt University, Berlin, for support as visiting guest scientists. All three authors thank Humboldt University for hospitality.}
\begin{document}
\begin{abstract} 
The  $c_2$ invariant of a Feynman graph is an arithmetic invariant   which detects many properties of the corresponding Feynman integral.
In this paper, we define the $c_2$ invariant in momentum space and prove that it equals the $c_2$ invariant in parametric space for overall log-divergent graphs.
Then we show that the $c_2$ invariant of a graph vanishes whenever it contains subdivergences.
Finally, we investigate how the $c_2$  invariant  relates to   identities such as the four-term relation in knot theory.
\end{abstract}
\maketitle

\section{Introduction}
Let $G$ be a connected graph. The graph polynomial of $G$ is defined by associating a variable $x_e$ to every edge $e$ of $G$ and setting
\begin{equation}\label{0}
\Psi_G(x)=\sum_{T\,\rm span.\,tree}\;\prod_{e\not \in T}x_e,
\end{equation}
where the sum is over all spanning trees $T$ of $G$. These polynomials first appeared in Kirchhoff's work on currents in electrical networks \cite{KIR}.

Let $N_G$ denote the number of edges of $G$, and let   $h_G$ denote the number of  independent cycles in $G$ (the first Betti number). 
Of particular interest is the case when $G$ is primitive and overall  logarithmically divergent:
\begin{eqnarray} \label{defnprimdiv}  N_G  & = & 2h_G  \\
N_{\gamma} &> &2 h_{\gamma} \quad \hbox{ for all strict non-trivial subgraphs } \gamma \subsetneq G \ .\nonumber
\end{eqnarray}
For such graphs, the corresponding Feynman integral (or residue) is independent of the choice of   renormalization scheme and  
can be defined by the following   convergent integral  in parametric space  (\cite{BEK}, \cite{Wein})
\begin{equation}\label{2}
I_G=\int_0^\infty\cdots\int_0^\infty\frac{ dx_1\cdots dx_{N_G}}{\Psi_G(x)^2}\,  \delta(\sum_{i=1}^{N_G} x_i-1).
\end{equation}
The numbers $I_G$  are notoriously difficult to calculate, and have been investigated intensively from the numerical \cite{BK, CENSUS}
and  algebro-geometric points of view \cite{BEK, BrFeyn}.   For   graphs in $\phi^4$ theory with subdivergences, the renormalised amplitudes can also be written in terms of graph polynomials
 by  subtracting  counter-terms from the  same leading term $\Psi^{-2}_G(x)$ \cite{BrK}.

Given the difficulty in computing $I_G$,  one seeks  more efficient  ways to extract qualitative information about the Feynman integral indirectly.
The motivic philosophy suggests studying the graph hypersurface: 
$$\overline{X}_G  \subset \PP^{N_G-1} $$
defined by the zero locus  of the graph polynomial $\Psi_G$ in projective space (with no restriction on the numbers of edges or cycles in $G$). 
In particular, motivated by a conjecture of Kontsevich \cite{KONT} (disproved for general graphs in \cite{BB}),  one can consider the point counting function 
$$q\mapsto |\overline{X}_G(\FF_q)|$$
where $q=p^n$ is a prime power, and $\FF_q$ is the finite field with $q$ elements. In \cite{BS},
it was  shown that for graphs with at least three vertices there is a map
$$c_2 : \{\hbox{graphs with } \geq 3 \hbox{ vertices}\} \rightarrow \prod_{{\rm prime\,powers}\,q}\ZZ/q\ZZ$$
such that, writing
$[X_G]_q:=|X_G(\FF_q)|,$ we have
\begin{equation}\label{c2modq2}
[X_G]_q \equiv c_2(G)_q \,  q^2 \mod q^3
\end{equation}
where $X_G\subset \A^{N_G}$ is the affine graph hypersurface given by the zero locus of $\Psi_G$, and $c_2(G)_q$ is itself the point counting function
on a related hypersurface.  One of the motivations for studying the $c_2$ invariant is the following conjecture, verified for all graphs with $\leq14$ edges,
which states that it only depends on  the residue of $G$ whenever it is defined.
\begin{con}\label{con1}
If $I_{G_1}=I_{G_2}$ for two primitive log-divergent graphs $G_1$, $G_2$ (i.e. which satisfy $(\ref{defnprimdiv})$) then $c_2(G_1) = c_2(G_2)$.
\end{con}

Furthermore,  for graphs $G$ which evaluate to multiple zeta values,  we expect  the residue $I_G$  to drop in transcendental weight   if and only if   $c_2(G)_q$ is identically zero \cite{BY}.  
 All $c_2$ invariants of primitive log-divergent graphs with $\leq20$ edges are listed for the first six primes in \cite{modQFT}.

\subsection{Avatars of the $c_2$ invariant}
Before stating our main results, it will be helpful to  discuss  various different incarnations of the $c_2$ invariant.

\vspace{0.05in}
\emph{1. Geometric}.   If $k$ is a field, we can consider the   class $[X_G]$ of the affine graph hypersurface in the Grothendieck ring of varieties $K_0(\mathrm{Var}_k)$ over $k$.
Let $\Lef=[\A^1_k] \in K_0(\mathrm{Var}_k)$ be the equivalence class of the affine line. Whenever $G$ has at least three vertices, in other words whenever $N_G\geq h_G+2$,  it was shown in \cite{BS} that there exists an element
\begin{equation}\label{c2geom} c_2(G) \in K_0(\mathrm{Var}_k)/\Lef\ ,
\end{equation}
given explicitly by the class of a certain hypersurface, such that
\begin{equation}\label{1}
[X_G] \equiv c_2(G) \Lef^2 \mod \Lef^3\ .
\end{equation}
This  is a refined version of  equation $(\ref{c2modq2})$.

\vspace{0.05in}

\emph{2. Arithmetic}. For some applications,  and for numerical computations, it is often simpler to restrict the point counting function to the fields $\FF_p$ of prime order only.
Thus in \cite{BS} we considered  the  vector
\begin{equation}\label{c2total}
\widetilde{c}_2(G) = ( c_2(G)_2,  c_2(G)_3,c_2(G)_5,\ldots ) \in \prod_{p\,   prime} \ZZ/p\ZZ\ . 
\end{equation}
It clearly  factors through the class  $(\ref{c2geom})$, but in many cases we are not able  to lift computations of $\widetilde{c}_2(G)$ to the
Grothendieck ring. In \cite{BS} we gave examples of graphs  such that $\widetilde{c}_2(G)$ is given by the Fourier coefficients of a modular form, giving 
explicit counter examples to  Kontsevich's conjecture. Several more modular $\widetilde{c}_2$ invariants were found in \cite{modQFT}.

\vspace{0.05in}

\emph{3. Analytic}.  It turns out that for a large class of graphs, one can in principle compute  the residue $I_G$   by integrating  in parametric space \cite{BrFeyn}.
After integrating out a subset of  edge variables $x_1,\ldots, x_n$  in $(\ref{2})$, one typically obtains an expression  whose  numerator  is an iterated integral
in the sense of K. T. Chen \cite{Ch1}, and whose denominator  is  a polynomial
$$D^n_G(x_1,\ldots, x_n) \in \ZZ[x_{n+1},\ldots, x_{N_G}]$$
of degree at most two in each variable.  When $D^n_G$ factorizes, one can define $D^{n+1}_G$  to be the resultant of its factors  with respect to $x_{n+1}$.
This sequence of polynomials (which terminates when $D^n_G$ can no longer be reduced) is called the denominator reduction.  When $D^n_G$ exists, we showed in \cite{BS} that  
\begin{equation}\label{introDn2c2}
c_2(G)_q \equiv (-1)^n [D^n_G]_q \mod q
\end{equation}
when $2h\leq N_G$ and $5\leq n<N_G$, as  a  consequence of the Chevalley-Warning theorem. This gives an effective way to compute $c_2(G)_q$ 
and is the main method for proving properties of  the $c_2$ invariant.

\vspace{0.05in}

\emph{4. Motivic}. We expect the $c_2$ invariant to relate to the framing of the graph motive \cite{BEK} given by the Feynman differential form (the integrand of $(\ref{2})$).  This partly justifies conjecture $\ref{con1}$ and the incarnations 1,2,3 above.

\vspace{0.05in}
Hereafter, we shall loosely refer to the $c_2$ invariant as any of the variants $1-3$  above, since  our results relate to all three different versions. As a result, there is considerable interplay  between  geometric, combinatorial, and arithmetic arguments throughout this paper.

\subsection{Results} 
It will be convenient to make the following definition.
\begin{defn} \label{defncgeneral} Let $X$ be a  scheme of finite type over $\mathrm{Spec} \,  \ZZ$.  Let $n\geq 0$. We say that $X$ has a $c_n$ invariant if
$[X]_q \equiv 0 \mod q^n $ for all  prime powers $q$. In this case, define the $c_n$ invariant of $X$ to be the function:
$$c_n(X) \equiv [X]_q/q^n \mod q$$
from the set of prime powers $q$ to $\ZZ/q \ZZ$.
\end{defn} 
Our  results are of three different  types.

\subsubsection{A $c_2$ invariant in momentum space}
Our first goal is to show that the $c_2$ invariant is intrinsic and is not simply a feature of the choice of integral representation for the Feynman graph.
For this, we define a momentum space representation for the $c_2$ invariant as follows.

The Feynman integral in momentum space is the integral of an algebraic differential form with singularities along a union  of quadrics
$Q_1,\ldots, Q_{N_G}$. With an appropriate choice of space-time metric, we  show that  the scheme $ V(Q_1\ldots Q_{N_G})$, which is defined over $\ZZ$, 
has a $c_2$ invariant  in the sense of definition $\ref{defncgeneral}$, and  we define $c_2^{\rm mom}(G)$ to be $c_2(V(Q_1\ldots Q_{N_G}))$.
In other words, we have the equation:
$$c_2(G)^{\rm mom}_q \equiv[Q_1Q_2\cdots Q_{N_G}]_q/q^2\mod q.$$
The first  result is that the $c_2$ invariant in momentum space is the same as the $c_2$ invariant in parametric space for logarithmically divergent graphs.
\begin{thm}Let $G$ be a graph with $h_G\geq3$. If $2h_G=N_G$ then 
$$c_2(G)^{\rm mom}_q \equiv c_2(G)_q\mod q.$$
\end{thm}

The proof requires studying the singular locus $\Sing(X_G)$  of $X_G$, and in particular proving the following intermediate result.

\begin{thm} If $G$ has at least 3 vertices, then $\Sing(X_G)$ has a $c_1$ invariant.
\end{thm}
This suggests  studying the $c_n$ invariants of the singular locus  of $X_G$ in  its own right.
In fact, we believe that the $c_1$ invariant of $\Sing(X_G)$ should vanish, in which case one could define its $c_2$ invariant, which we expect to be non-zero in general.
This would give a new graph invariant  $c^{\mathrm{sing}}_2(G)= c_2(\Sing\,  X_G)$,  which would be interesting to understand combinatorially.

\subsubsection{Vanishing for subdivergences} The second set of results extends our previous work on criteria for graphs to have weight drop \cite{BY}. 
 
\begin{thm}  \label{thmintrowd}
Let $G$ be an overall logarithmically divergent graph in $\phi^4$ theory. If $G$ has a non-trivial divergent subgraph then $c_2(G)_q=0$.
\end{thm}

Such a graph $G$ with a subdivergence can always be written as a $2,3$, or $4$-edge join. In the first two cases, we  prove that
$c_2(G)$ vanishes in the Grothendieck ring, but the case of a $4$-edge join is more subtle and we can only show the result on the level of point counting functions.
If reduced denominators $D^n_G$ exist for the elimination of all edges of the subdivergence, then in the last non-trivial step $D^n_G$ equals the square of the graph polynomial
of $G$ with fully contracted subdivergence. This explains the vanishing of the $c_2$ invariant on the level of denominator reduction.

Given the expected relation between vanishing $c_2$ and transcendental weight drop of Feynman amplitudes,   theorem  $\ref{thmintrowd}$ is evidence for   a folklore conjecture which states that the highest weight part of the lowest logarithmic power of the renormalised amplitudes in $\phi^4$ theory  
is independent of the choice of renormalisation scheme.

\subsubsection{Combinatorial identities}  In the light of conjecture $\ref{con1}$, and the many observed but unexplained algebraic relations between residues of Feynman graphs,
an important question is to understand precisely which combinatorial information is contained in the $c_2$ or related  invariants.

For a list of currently known or conjectured properties of the $c_2$ invariant, see \cite{BS}, \S4. To these can be added  some further relations for the denominator reduction described in
\cite{BY}, \S4.4-4.7, which immediately imply identities for the $c_2$ invariant via $(\ref{introDn2c2})$. Although an overarching combinatorial explanation for all these identities is
still lacking,  in $\S6$ we  describe  some new additive properties of  denominator polynomials which give a single  explanation for many of the
identities of \cite{BY}.

Finally, there remains the question of trying to relate $c_2(G)$ to other classical invariants in the theory of graphs. A tantalizing but  mysterious connection between
knots and Feynman integrals was investigated  by Broadhurst and Kreimer in the 90's \cite{BK, B3, BK1, K1}, but has proven very hard to verify in concrete cases because of the
difficulty in computation of Feynman integrals, and the high loop orders of the diagrams involved. The $c_2$ invariant provides us with a  tool to investigate such identities
without having to compute any integrals.

In this paper, we investigated the  4-term relation for chord diagrams, which was shown  to hold in some cases in \cite{bk4tr}, but found no such relation on the level of $c_2$ invariants
in  $\phi^4$ theory.  To our surprise, however,   we found  that the 4-term identity actually holds true on the level of the denominator polynomials $D^7_G$.

\section{Reminders on graph polynomials}
For the convenience of the reader, we gather some of the results on graph polynomials and various auxiliary polynomials to be used later.

\subsection{Graph matrix}
Let $G$ be any graph. We will use the following matrix representation for the   graph polynomial. 

\begin{defn} Choose an orientation on the edges of $G$, and for every edge $e$ and vertex $v$ of $G$, define the incidence matrix:
$$(\mathcal{E}_G)_{e,v} = \left\{
                           \begin{array}{rl}
                             1, & \hbox{if the edge } e \hbox{ begins at } v \hbox{ and does not end at } v,\\
                             -1, & \hbox{if the edge } e  \hbox{ ends at } v \hbox{ and does not begin at } v,\\
                             0, & \hbox{otherwise}.
                           \end{array}
                         \right.
 $$
Let $A$ be the diagonal matrix with entries  $x_e$, for $e \in E(G)$, and set
$$\widetilde{M}_G=\left(
  \begin{array}{c|c}
    A  & \mathcal{E}_G  \\
    \hline
  {-}\mathcal{E}_G^T&  0  \\
  \end{array}
\right)
$$
where the first $N_G$ rows and columns are indexed by the set of edges of $G$, and the remaining $v_G$ rows and columns are indexed by the set of vertices of $G$, in some order.
The matrix $\widetilde{M}_G$ has corank  $\geq 1$.  Choose any vertex of $G$ and let $M_G$ denote the square $(N_G+v_G-1)\times (N_G+v_G-1)$ matrix obtained from it by deleting
the row and column indexed by this vertex.
\end{defn}

It follows from the matrix-tree theorem that the graph polynomial satisfies 
$$\Psi_G=\det (M_G)\ .$$
This formula implies that $\Psi_G$ vanishes if $G$ has more than one component.

\subsection{Dodgson polynomials} \label{subsectDodg}

We use the following notation.
\begin{defn}
If $f=f_1+f^1 x_1$ and $g=g_1+g^1 x_1$ are polynomials of degree one in $x_1$, recall that their resultant is defined by:
\begin{equation} \label{resultantdef}
[f,g]_{x_1}= f^1g_1-f_1g^1\ .
\end{equation}
\end{defn}

\begin{defn}  Let $I,J, K$ be subsets of the set of edges of $G$ which satisfy $|I|=|J|$. Let  
$M_G(I,J)_K$ denote the matrix obtained from $M_G$ by removing the rows indexed by the set $I$ and columns indexed by the set $J$, and setting $x_e=0$ for all $e\in K$. 
Let 
\begin{equation} \label{Dogsondefn} \Psi_{G,K}^{I,J}=\det M_G(I,J)_K\ . \end{equation}
\end{defn}
 We  write $\Psi^I_{G,K}$ as a shorthand for $\Psi^{I,I}_{G,K}$ and drop the subscript $K$ if it is empty.
Since the matrix $M_G$ depends on various choices, the polynomials $\Psi^{I,J}_{G,K}$ are only well-defined up to sign. In what follows, for any graph $G$, we shall fix a particular
matrix $M_G$ and this will fix all the signs in the polynomials $\Psi^{I,J}_{G,K}$ too.

We now state some identities between Dodgson polynomials which will be used in the sequel. The proofs can be found in (\cite{BrFeyn}, \S2.4-2.6).

\begin{enumerate}
\item \emph{The contraction-deletion formula}. The graph polynomial  is linear in its variables and fulfills the contraction-deletion relation
\begin{equation}\label{1a}
\Psi_G = \Psi_{G\backslash e} x_e + \Psi_{G\q e}\ ,
\end{equation}
where the graph polynomial of disconnected graphs is zero. Likewise the contraction ($\q$) of a self-loop is zero in the graph algebra and $\Psi_0=0$.
More generally, if $|I|=|J|$, we have:
  $$\Psi^{Ie,Je}_{G,K}=\pm \Psi^{I,J}_{G\backslash e, K} \, \hbox{ and }  \Psi^{I,J}_{G,Ke} = \pm \Psi^{I,J}_{G\q e, K}\ . $$
 \vspace{0.05in}
\item \emph{Dodgson identities}. Let $I,J$ be two subsets of edges of $G$ such that $|I|=|J|$ and let $a,b,x\notin I \cup J
\cup K$ with $a,b<x$  (or $x<a,b$
). The  first identity is:
$$ \big[\Psi^{I,J}_{G,K}, \Psi^{Ia,Jb}_{G,K} \big]_x = \Psi^{Ix,Jb}_{G,K} \Psi^{Ia,Jx}_{G,K}\ .$$
 Let $I,J$ be two subsets of edges of $G$ such that $|J|=|I|+1$ and let $a,b,x\notin I\cup J
\cup K$ with $x<a<b$. Then the second identity is:
$$ \big[\Psi^{Ia,J}_{G,K}, \Psi^{
Ib,J}_{G,K} \big]_x = - \Psi^{Ix,J}_{G,K} \Psi^{Iab,Jx}_{G,K}\ . $$
\vspace{0.05in}
\item \emph{Pl\"ucker identities}. Let $i_1<i_2<i_3<i_4$. Then
$$\Psi^{i_1 i_2,i_3i_4}_G -  \Psi^{i_1i_3,i_2i_4}_G+\Psi^{i_1 i_4,i_2i_3}_G=0\ .$$
For an increasing sequence of edges $i_1<\ldots < i_6$ we have
$$\Psi^{i_1 i_2 i_3,i_4i_5i_6}_G - \Psi^{i_1i_2i_4,i_3i_5i_6}_G+ \Psi^{i_1i_2i_5,i_3i_4i_6}_G- \Psi^{i_1i_2i_6,i_3i_4i_5}_G=0\ .$$

\item \emph{Vanishing}.
Suppose that $E=\{e_1,\ldots, e_k\}$ is the set of  edges which are adjacent to a given  vertex of $G$. Then
$\Psi^{I,J}_{G,K} =0$ if  $E\subset I$ or $E \subset J$.
Now suppose that $E=\{e_1,\ldots, e_k\}$ is a set of edges in $G$ which contain a cycle. Then 
$\Psi^{I,J}_{G,K} =0$ if $( E\subset I\cup  K$  or  $E\subset J \cup K)$   and $E \cap I \cap J= \emptyset$.

\end{enumerate}

\subsection{Spanning forest  polynomials}
Dodgson polynomials are in turn linear combinations of more basic polynomials, called spanning forest polynomials \cite{BY}.

\begin{defn} 
Let $X$
be a set of vertices of $G$, and let $P = \{P_1, \ldots, P_k\}$ be a  partition of  $X$.   Define the spanning forest polynomial by
\[
    \Phi_G^P = \sum_F\prod_{e \not\in F} x_e
\]
where the sum runs over spanning forests $F=T_1\cup \ldots \cup T_k$ where each tree $T_i$ (possibly a single vertex) of $F$ 
contains the  vertices in $P_i$ and no other vertices of $X$. Thus $V(T_i) \supseteq P_i$ and $V(T_i)\cap P_j =\emptyset $ for  $j\neq i$.  
\end{defn}

We represent $\Phi_G^P$ by associating a colour to each part of $P$ and drawing $G$ with the vertices in $X$ coloured accordingly.

\begin{prop}\label{exists}
Let $I, J$ be sets of edges of $G$ with $|I| = |J|$ and $I \cap J = \varnothing$.  Then we can write
  \[
    \Psi_G^{I,J} = \sum_{i} f_i \Phi_G^{P_i}
  \]
where the sum runs over partitions of $V(I\cup J)$ and $f_i \in \{-1, 0, 1\}$.  In particular, $f_i\neq 0$ precisely when each forest consistent with $P_i$ becomes
a tree in $G/I\backslash J$ and in $G/J\backslash I$.
\end{prop}
Note that the sign $f_i$ can be computed by taking any forest $F$ consistent with $P_i$ and then considering the determinant of the  matrix obtained from $M_G$ by removing rows and columns
indexed by the set of edges not in $F$.  This determinant reduces \cite{BY} to
\begin{equation} \label{detmatrices}
\det [E_I N] \det[E_J N]
\end{equation}
where $E_I$ is the matrix of the columns corresponding to edge indices $I$ of $\mathcal{E}_G$ with one row removed, likewise for $E_J$, and $N$ is the matrix of columns
corresponding to edges of $G\backslash (I\cup J)$ which do not appear in the forest $F$.

\subsection{Denominator reduction}
\begin{defn} \label{Def5inv} Let $i,j,k,l,m$ be  five distinct edges in  $G$. The five-invariant  of these edges is the polynomial defined up to a sign by the resultant
$${}^5\Psi_G(i,j,k,l,m) =  \pm[\Psi^{ij,kl}, \Psi^{ik,jl}]_m \ .$$
Permuting the order of the edges $i,j,k,l,m$ only affects the overall sign.
\end{defn}
Denominator reduction is the name given to the elimination of variables by taking iterated resultants, starting with the 5-invariant.
Let $G$ be a graph, and order its edges $1, \ldots, {N_G}$.  Set  $D^5_G(1,\ldots, 5) =\pm {}^5 \Psi_G(1,\ldots, 5)$, and define 
a sequence of polynomials (conditionally) as follows. 

\begin{defn} 
Let  $n\geq 5$ and suppose that $D^n_G(1,\ldots, n)$ is defined, and further  that it factorizes into a product of  factors $f,g$ of degree $\leq 1$   in $x_{n+1}$. Then set 
$${D}^{n+1}_G(1,\ldots, {n+1}) =\pm [  f,g]_{n+1} \ ,$$
We say that $G$ is denominator reducible if there exists an order of edges such that ${D}^n_G(1,\ldots, {n})$ is defined for all $n$.
We say that $G$ has weight drop if there exists an order of edges such that ${D}^n_G(1,\ldots, {n})$ vanishes for some $n$.
\end{defn}

The relation between the denominator reduction and  $c_2$ invariant is given by the following theorem (theorem 29 in \cite{BS}).

\begin{thm} \label{thmDenom}
Let $G$ be a connected graph with  $2h_G\leq N_G$. Suppose that $D^n_G(e_1,\ldots, e_n)$ is the result of
the denominator reduction after $5\leq n<N_G$ steps. Then
\begin{equation}\label{3}
c_2(G)_q \equiv  (-1)^n [D^n_G(e_1,\ldots, e_n)]_q \mod q\ .
\end{equation}
\end{thm}

\section{The $c_2$ invariant in momentum space}
For any  primitive log-divergent  graph $G$,  the residue $I_G$ of $G$ can be written as an integral in various different representations.  From a physical point of view,
the most natural of these  is the representation of $I_G$ as an integral in momentum space \cite{CENSUS}. Other possibilities are parametric space as explained in the introduction,
position space, related to momentum space by a Fourier transform,  and dual parametric space which is linked to the parametric
formulation (\ref{2}) by inversion of the Schwinger coordinates  $x_e$. In the spirit of conjecture \ref{con1} for graphs which have a residue,  all these representations should
lead to equivalent $c_2$ invariants.

Because we work over a general field $k$ which does not necessarily contain $\sqrt{-1}$ or may have characteristic 2 the choice of metric becomes relevant
for the definition of Feynman rules in momentum and in position space. Here it is best to use a twistor type metric with signature $(+,-,+,-)$. We choose
the metric $\eta$ to be of the form
\begin{equation}\label{8}
\eta=\left(\begin{array}{cccc}0&1&0&0\\
1&0&0&0\\
0&0&0&1\\
0&0&1&0\end{array}\right)
\end{equation}
and write $p=(p^+,p^-,p'^+,p'^-)$. 
Then the propagator of a massless particle becomes $1/Q(p)$ with (see \cite{IZ})
\begin{equation} \label{Qasp}
Q(p)=p^+p^-+p'^+p'^-, 
\end{equation}
which is linear in the coordinates. 
The value of the residue does not depend on the chosen metric. Physically this means that the residue is a scalar.

Likewise, in position space the propagator between $x$ and $y$ in $k^4$ is $1/Q(x-y)$.

In the following we focus on momentum space. We fix a basis of $h_G$ independent cycles in $G$ with respect to which the
momenta $p=(p_1,\ldots,p_{h_G})$ are routed. The graph $G$ has $N_G$ edges with propagators $1/Q_1(p),\ldots,1/Q_{N_G}(p)$.
We will show that the `Schwinger trick' lifts to the $c_2$ invariant proving the existence of a $c_2$ invariant in momentum space if $2h_G\geq N_G$ and
its equivalence with (\ref{c2modq2}) for log-divergent graphs. 

\begin{ex}\label{ex1}
\begin{figure}[ht]
\epsfig{file=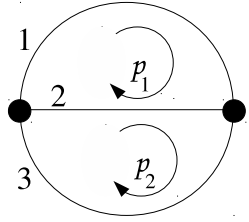,width=\textwidth}
\caption{The sunset graph.}
\end{figure}
We consider the sunset graph in fig.\ 1.
The edges 1,2,3 have the propagators $1/Q_1,1/Q_2,1/Q_3$ with
\begin{eqnarray*}
Q_1&=&p_1^+p_1^-+p_1'^+p_1'^-\\
Q_2&=&(p_1^+-p_2^+)(p_1^--p_2^-)+(p_1'^+-p_2'^+)(p_1'^--p_2'^-)\\
Q_3&=&p_2^+p_2^-+p_2'^+p_2'^-
\end{eqnarray*}
An explicit computer calculation using Stembridge's reduction \cite{Stem} yields
$$[Q_1Q_2Q_3]=3\Lef^7-7\Lef^5+4\Lef^4+4\Lef^3-3\Lef^2.$$\
The momentum space $c_2$ invariant exists (see proposition-definition \ref{def1} below) and is equal to  the $\Lef^2$ coefficient of $[Q_1Q_2Q_3]$, namely $-3\mod\Lef$.
\end{ex}

 The key tool in the Schwinger trick is the universal quadric
\begin{equation}
\sQ(x,p)=x_1Q_1(p)+x_2Q_2(p)+\ldots + x_{N_G}Q_{N_G}(p).
\end{equation}
From the matrix-tree theorem used in the Schwinger trick \cite{IZ} we conclude that there exists a symmetric $h_G\times h_G$ matrix $N$ such that
\begin{eqnarray}\label{9}
\sQ(x,p)&=&(p^-,p'^-)\left(\begin{array}{cc}N(x)&0\\
0&N(x)\end{array}\right)\left(\begin{array}{c}p^+\\
p'^+\end{array}\right),\quad\hbox{with}\\
\hbox{det}\,N(x)&=&\Psi_G(x).
\end{eqnarray}
Here $p^\pm=(p^\pm_1,\ldots,p^\pm_{h_G})$ and likewise $p'^\pm$.
\begin{prop}\label{prop1}
\begin{enumerate}
\item
The singular locus of $X_G$ is given by
\begin{equation}\label{10}
\Sing(X_G)=\{x:\rank N(x)<h_G-1\}.
\end{equation}
\item
Let $I\subseteq\{1,\ldots,N_G\}$ and $N_{\bar{I}}(x)=N(x)|_{x_k=0,\,{\rm if}\,k\not\in I}$ be obtained from $N$ by setting all variables to zero whose index is not in $I$. Then
\begin{equation}\label{11}
(\Lef-1)\Lef^{|I|-1}[Q_{i,i\in I}]=(\Lef-1)\Lef^{2h_G-1}[N_{\bar{I}}\cdot p^+,N_{\bar{I}}\cdot p'^+].
\end{equation}
\item
With $\Psi_{G,{\bar{I}}}(x)=\Psi_G(x)|_{x_k=0,\,{\rm if}\,k\not\in I}$ we have
\begin{equation}\label{12}
[N_{\bar{I}}\cdot p^+,N_{\bar{I}}\cdot p'^+]\equiv(\Lef^2-1)[\Psi_{G,{\bar{I}}}]-\Lef^2[\rank N_{\bar{I}}<h_G-1]+\Lef^{|I|}\!\mod\Lef^4.
\end{equation}
\end{enumerate}
\end{prop}
\begin{proof}
\begin{enumerate}
\item 
With elementary row and column transformations (which correspond to a change of cycle basis) we can transform $N$ into a matrix $\tilde N$ with the property that
in each diagonal entry $\tilde{N}_{i,i}$ there exists a variable (say $x_i$) which does not occur in any other entry of $\tilde{N}$.
Because elementary row and column transformations preserve the rank of the matrix we may assume without restriction that $N=\tilde{N}$.
Let $x\in\hbox{Sing}\,(X_G)$. 
We thus have $\partial_{x_i}\Psi_G(x)=\det N^{i,i}(x)=0$ for $i=1,\ldots,h_G$ where $N^{J,K}$ is the matrix $N$ with rows $J$ and columns $K$ deleted.
We have $\det N(x)=0$,  and  the Dodgson identity for the symmetric matrix $N$:
$$(\det N^{i,j})^2=\det N^{i,i}\det N^{j,j}-\det N\det N^{ij,ij},$$
implies $\det N(x)^{i,j}=0$ for all $i,j=1,\ldots,h_G$. Hence $\rank N<h_G-1$.

On the other hand, if $\rank N(x)<h_G-1$ then $\det N(x)^{i,j}=0$ for all $i,j=1,\ldots h_G$, and in particular   $\partial_{x_i}\Psi_G(x)=\det N^{i,i}(x)=0$ for $i=1,\ldots,h_G$.  Hence  $x\in\Sing(X_G)$ and (\ref{10}) is established.
\item
Consider the universal quadric $\sQ_I=\sum_{i\in I}x_iQ_i$ and calculate its class in the Grothendieck ring in two different ways.

Firstly, $\sQ_I$ defines a family of hyperplanes in the $|I|$ dimensional affine space $\A^{|I|}$ with coordinates $x_i$.
Consider the fiber of the projection $V(\sQ_I) \rightarrow \A^{4h_G}$. In the generic case it is a hyperplane  in $\A^{|I|}$ whose class is $\Lef^{|I|-1}$.
Otherwise, all $Q_i,i\in I$ vanish and the fiber is $\A^{|I|}$. We have
$$[\sQ_I] = \Lef^{|I|-1}(\Lef^{4h_G}-[Q_{i,i\in I}]) + \Lef^{|I|}[Q_{i,i\in I}].$$

Secondly, from (\ref{9}) we have
$$\sQ_I(x,p)=(p^-,p'^-)\left(\begin{array}{cc}N_{\bar{I}}(x)&0\\
0&N_{\bar{I}}(x)\end{array}\right)\left(\begin{array}{c}p^+\\
p'^+\end{array}\right)$$
and so $\sQ_I$ also defines a family of hyperplanes in the $p^-$ variables. We now consider the fiber of the projection $V(\sQ_I) \rightarrow \A^{|I|+2h_G}$
and obtain
$$[\sQ_I] = \Lef^{2h_G-1}(\Lef^{|I|+2h_G}-[N_{\bar{I}}\cdot p^+,N_{\bar{I}}\cdot p'^+]) + \Lef^{2h_G}[N_{\bar{I}}\cdot p^+,N_{\bar{I}}\cdot p'^+].$$
Together we obtain (\ref{11}).
\item
The equations $N_{\bar{I}}\cdot p^+,N_{\bar{I}}\cdot p'^+$ form two identical systems of $h_G$ linear equations in the variables $p^+$ and $p'^+$, respectively.
The vanishing locus of each system is $\A^n$ where $n=\corank(N_{\bar{I}})$. Hence
\begin{eqnarray*}
[N_{\bar{I}}\cdot p^+,N_{\bar{I}}\cdot p'^+]&\equiv&[\corank N_{\bar{I}}=0]+\Lef^2[\corank N_{\bar{I}}=1]\mod\Lef^4\\
&\equiv&\Lef^{|I|}-[\corank N_{\bar{I}}>0]\\
&&+\,\Lef^2([\corank N_{\bar{I}}>0]-[\corank N_{\bar{I}}>1])\mod\Lef^4
\end{eqnarray*}
Because $\corank N_{\bar{I}}>0\Leftrightarrow\Psi_{G,{\bar{I}}}=0$ we obtain (\ref{12}).
\end{enumerate}
\end{proof}
To progress further we pass to finite fields. Let $q=p^n$ be a prime power. Given polynomials $P_1,\ldots, P_{\ell}\in \ZZ[x_1,\ldots, x_N]$, let  
$$[P_1,\ldots, P_{\ell}]_q \in \mathbb{N}\cup \{0\}$$
denote the number of points on the affine variety $V(\overline{P_1},\ldots, \overline{P_{\ell}}) \subset \FF_q^N$, where $\overline{P}_i$ denotes the reduction of $P_i$ modulo $p$.
The point-counting is compatible with inclusion-exclusion and Cartesian products and therefore factors through the Grothendieck ring mapping $\Lef$ to $q$.

We are interested in the point-count of the zero locus  of the denominator of the  momentum space differential form which is $[Q_1Q_2\cdots Q_{N_G}]_q$.
\begin{prop}\label{prop2}
Let $h_G\geq2$, $2h_G\geq N_G$ and $E$ be  the edge-set of $G$, then
\begin{eqnarray}\label{14}
[Q_1Q_2\cdots Q_{N_G}]_q& \equiv&(-q)^{2h_G-N_G}\Big[[\Psi_G]_q+q^2[\Sing(X_G)]_q\nonumber\\
&&-\,q\sum_{e\in E}[\Psi_{G\q e}]_q+q^2\sum_{e_1,e_2\in E}[\Psi_{G\q e_1e_2}]_q\Big]\mod q^3.
\end{eqnarray}
\end{prop}
\begin{proof}
By inclusion exclusion we obtain
$$[Q_1Q_2\cdots Q_{N_G}]_q=\sum_{\emptyset\neq I\subseteq\{1,\ldots,N_G\}}(-1)^{|I|-1}[Q_{i,i\in I}]_q.$$
By prop.\ \ref{prop1} (2) we have $[Q_{i,i\in I}]_q=q^{2h_G-|I|}[N_{\bar{I}}\cdot p^+,N_{\bar{I}}\cdot p'^+]_q$.
Next, we use prop.\ \ref{prop1} (3). For $2h_G\geq N_G$ the second term on the right hand side of (\ref{12}) survives mod $q^3$ only in the case $I=\{1,\ldots,N_G\}$
where it gives $-q^2[\Sing(X_G)]_q$ by prop.\ \ref{prop1} (1). The third term on the right hand side of (\ref{12}) is multiplied by $q^{2h_G-|I|}$ and vanishes
 mod $q^3$.
The first term on the right hand side of (\ref{12}) vanishes mod $q^3$ unless $|I|\geq N_G-2$ and in this case  gives $(q^2-1)[\Psi_{G,{\bar{I}}}]_q=(q^2-1)[\Psi_{G\q (\{1,\ldots,N_G\}-I)}]_q$ by eq.\ (\ref{1a}).
The term proportional to $q^2$ vanishes trivially for $|I|<N_G$ or $2h_G>N_G$. If $|I|=N_G$ and $2h_G=N_G$ then from $h_G\geq2$ it follows
that $h_G+2\leq N_G$. In this case we know from (\ref{c2modq2}) that $q^2|[\Psi_G]_q$ with the result  that the $q^2$-term vanishes mod $q^3$. Putting everything together proves (\ref{14}).
\end{proof}

The above proposition allows us to define the $c_2$ invariant in momentum space
\begin{propdef}\label{def1}
Let $G$ be a graph with $h_G\geq2$ independent cycles and $N_G\leq2h_G$ edges. Fix a cycle basis in $G$ and define the inverse propagators according to momentum space
Feynman rules with metric (\ref{8}). Then the momentum space $c_2$ invariant of $G$ is given as a map from $q$ to $\ZZ/q\ZZ$ by
\begin{equation}
c_2(G)^{\rm mom}_q \equiv[Q_1Q_2\cdots Q_{N_G}]_q/q^2\mod q.
\end{equation}
\end{propdef}
\begin{proof}
We use eq.\ (\ref{14}) to show that $[Q_1Q_2\cdots Q_{N_G}]$ is divisible by $q^2$.
If $2h_G=N_G$ then from $h_G\geq2$ we get $h_G+2\leq N_G$ with the result that  $q^2|[\Psi_G]_q$ by (\ref{c2modq2}). Moreover, we have
either $G\q e=0$ in the graph algebra or $h_{G\q e}+1\leq N_{G\q e}$. In any case $q|[\Psi_{G\q e}]_q$, see \cite{AM}.
If $2h_G=N_G+1$ then from $h_G\geq2$ we get $h_G+1\leq N_G$ and thus $q|[\Psi_{G}]_q$.
In all other cases $[Q_1Q_2\cdots Q_{N_G}]_q$ is trivially divisible by $q^2$.
\end{proof}
Note that the point-count  $[Q_1Q_2\cdots Q_{N_G}]_q$ is independent of the chosen cycle basis, since the change of cycle basis results in linear transformations
of the underlying coordinates.

\begin{thm}\label{thm1}
Let $G$ be a graph with $h_G\geq3$. If $2h_G>N_G$ then $c_2(G)^{\rm mom}_q \equiv0\mod q$. 

If $2h_G=N_G$ then the momentum space $c_2$ invariant equals
the $c_2$ invariant in parametric space modulo $q$,
\begin{equation}\label{15a}
c_2(G)^{\rm mom}_q \equiv c_2(G)_q\mod q.
\end{equation}
\end{thm}
\begin{proof}
We again use eq.\ (\ref{14}). If we subtract $2h_G-N_G=d$ from $h_G\geq3$ we obtain $N_G\geq h_G+3-d$.

If $d\geq3$ the statement of the theorem follows trivially.

If $d=2$ then $N_G\geq h_G+1$, hence $q|[\Psi_G]_q$, see \cite{AM}, and the theorem follows.

If $d=1$ then $N_G\geq h_G+2$, hence $q^2|[\Psi_G]_q$ by (\ref{1}) and $N_{G\q e}=0$ or $N_{G\q e}\geq h_{G\q e}+1$, hence $q|[\Psi_{G\q e}]_q$. Again, the theorem follows.

If $d=0$ then $N_{G\q e}=0$ or $N_{G\q e}\geq h_{G\q e}+2$, hence $q^2|[\Psi_{G\q e}]_q$. Likewise $N_{G\q e_1e_2}=0$ or $N_{G\q e_1e_2}\geq h_{G\q e_1e_2}+1$, hence $q|[\Psi_{G\q e_1e_2}]_q$.
In this case we obtain
$$[Q_1Q_2\cdots Q_{N_G}]_q \equiv q^2(c_2(G)_q+[\Sing(X_G)]_q)\mod q^3.$$
The theorem follows from $[\Sing(X_G)]_q \equiv0\mod q$ for graphs with $N_G\geq h_G+2$ which we will prove in thm.\ \ref{thmsing}.
\end{proof}
Note that the residue $I_G$, see (\ref{2}), only exists in the case $2h_G=N_G$. Moreover, graphs with non-trivial residues have $h_G\geq3$.
In ex.\ \ref{ex1} we saw that we get a non-trivial $c_2(G)^{\rm mom}$ if $2h_G>N_G$ but $h_G<3$.

The $c_2$ invariant in parametric space does not in general vanish if $2h_G>N_G$.
We rather have $c_2(G)_q\equiv0\mod q$ if $2h_G<N_G$ and $N_G\geq4$, see \cite{BS}. For the sunset
graph which has $2h_G=N_G+1$ in ex. \ref{ex1} we obtain in parametric space $c_2(G)_q\equiv 1\not\equiv c^{\rm mom}_2(G)_q\equiv-3\mod q$.

It is possible to define a $c_2$ invariant in position space $c_2(G)^{\rm pos}_q$ and in dual parametric space $c_2(G)^{\rm dual}_q$.
If $2h_G=N_G$ both $c_2$ invariants can be shown to be equal mod $q$ by translating the methods of thm.\ \ref{thm1} to position space.
The equivalence of $c_2(G)^{\rm dual}_q$ and the $c_2$ invariant in parametric space is conjectured for graphs which have a residue in \cite{SchnetzFq}.
This has still  not been proved even though   dual parametric space and parametric space are only related by inversion of variables.

It is important to note that only in the case $2h_G=N_G$ (the case in which the residue exists) are all $c_2$ invariants  (conjecturally) equivalent. In this case
the information contained in  the  various $c_2$ invariants is carried by the graph itself rather than by any of the representations of the residue integral (\ref{2}).

\section{The singular locus of graph hypersurfaces}
Let $G$ be a connected graph with edge-set $E(G)$, and let $X_G$ denote its graph hypersurface.
By linearity of the graph polynomial, the partial derivatives satisfy
$${\partial\Psi_G \over \partial x_e}  = \Psi^e_G \qquad \hbox{ for }  e\in E(G)\ .$$
The singular locus of $X_G$ is the affine scheme  $\Sing(X_G) = V(\Psi^e_G,  e\in E(G)).$
Let $[\Sing( X_G)]$ denote its class in $K_0(\mathrm{Var}_k)$, for $k$ a field. We shall prove:
\begin{thm}\label{thmsing}
Let $G$ be a graph with at least 3 vertices. Then
$$[ \Sing (X_G)] \equiv 0 \mod \Lef\ .$$
\end{thm}
In particular, $[ \Sing (X_G)]_q \equiv 0 \mod q\ $ for all prime powers $q$.
\begin{remark} We believe that $[ \Sing (X_G)] $ should be congruent to zero modulo  $\Lef^2 $ for all reasonable graphs. If so, then one can define the $c_2$ invariant of the singular locus, and one can ask if it is  related to the $c_2$ invariant of $X_G$.
\end{remark} 

\subsection{Preliminary identities}
The proof of theorem \ref{thmsing} requires some  elimination theory and some new identities between Dodgson polynomials.
For simplicity of notation, we drop the subscript $G$ throughout this section.

\begin{lem} Let $i,j,k$ denote any three distinct edges of $G$. Then
\begin{equation}\label{id2}
[\Psi^i, \Psi^j]_{k} = \Psi^{ij,ik}\Psi^{j,k} -  \Psi^{ij,jk}\Psi^{i,k} \ .
\end{equation}
\end{lem}
\begin{proof}
First observe that from the Dodgson identity and linearity:
$$\Psi_k^i\Psi^k_i - \Psi^{ik}\Psi_{ik}= (\Psi^{i,k})^2= (\Psi^{ij,jk} x_j + \Psi^{i,k}_j)^2\ .$$
Taking the coefficient of $x_j$ on both sides of this expression gives:
$$\Psi_k^{ij}\Psi^k_{ij} - \Psi^{ijk}\Psi_{ijk}+\Psi_{jk}^i\Psi^{jk}_i - \Psi_j^{ik}\Psi^j_{ik}= 2\, \Psi^{ij,jk}  \Psi^{i,k}_j\ .$$
Subtract the same expression with $i,j$ interchanged:
$$2( \Psi_{jk}^i\Psi^{jk}_i - \Psi_j^{ik}\Psi^j_{ik}) = 2 \Psi^{ij,jk}  \Psi^{i,k}_j- 2\Psi^{ij,ik}  \Psi^{j,k}_i$$
Rewriting the left-hand side as a resultant gives
\begin{equation} \label{weirdres}
{[}\Psi_{j}^i,\Psi^{j}_i]_k =   \Psi^{ij,ik}  \Psi^{j,k}_i- \Psi^{ij,jk}  \Psi^{i,k}_j \ .
\end{equation} 
Now we wish to compute
$$[\Psi^i, \Psi^j]_{k} =[\Psi^{ij}x_j +\Psi^i_j , \Psi^{ij}x_i + \Psi_i^j]_{k} = [ \Psi^i_j , \Psi^{ij} ]_k\, x_i + [\Psi^{ij}, \Psi_i^j]_k \,  x_j + [\Psi^i_j ,  \Psi_i^j]_{k} $$
By the Dodgson identity and $(\ref{weirdres})$ this reduces to
$$[\Psi^i, \Psi^j]_{k} =  (\Psi^{ij,ik})^2\, x_i - (\Psi^{ij,jk})^2\,   x_j +  \Psi^{ij,ik}  \Psi^{j,k}_i- \Psi^{ij,jk}  \Psi^{i,k}_j  $$
which, after writing $\Psi^{j,k}=\Psi^{ij,ik} \, x_i +    \Psi^{j,k}_i$ and likewise $\Psi^{i,k}$, is equation $(\ref{id2})$.
\end{proof}

\begin{cor}
Let $I$  denote  the ideal in $\QQ[x_e, e\in E(G)]$ spanned by $\Psi^k$ and $\Psi_k$. Then
\begin{equation} \label{id3}
[\Psi^i, \Psi^j]_{k} \in \sqrt{I} \quad  \hbox{ for all } i,j\in E(G)\ .
\end{equation}
\end{cor}
\begin{proof}
 The Dodgson identity and linearity give
$$(\Psi^{i,k})^2 =  [\Psi_i, \Psi^i]_k =  [\Psi, \Psi^i]_k = \Psi^k \Psi_k^{i} - \Psi_k \Psi^{ik} \in I\ .$$
It follows that  $\Psi^{i,k},\Psi^{j,k}\in \sqrt{I}$.   By  $(\ref{id2})$ this  gives  $[\Psi^i, \Psi^j]_{k} \in \sqrt{I}$.
\end{proof}

We say that  a subgraph  $\gamma\subseteq G$ is a cycle  if $\gamma$ is a topological  circle, i.e.,  $h_{\gamma}=1$ and $h_{\gamma\backslash e}=0$ for all $e\in E(\gamma)$. 

\begin{lem}\label{help for PsiGlower1} Let $1,\ldots, k$ be a cycle in $G$; let the vertex between edges $i$ and $i+1$ be $v_i$ and let the vertex between edges $1$ and $k$ be $v_k$.  Then
\begin{equation}\label{PsiGlower1 in Phi}
\Phi^{\{v_1\},\{v_k\}}_{H} = \Phi^{\{v_1\},\{v_2v_k\}}_{H}  + \sum_{j=3}^{k-1} \left(\Phi^{\{v_1v_{j-1}\},\{v_jv_k\}}_H - \Phi^{\{v_1v_j\},\{v_{j-1}v_k\}}_H\right)+ \Phi^{\{v_k\},\{v_1v_{k-1}\}}_H
\end{equation}
where $H=G\backslash 1\cdots k$.
\end{lem}

\begin{proof}
  The proof is by induction on the length of the cycle.

   Take $k\geq 4$.  Consider the first three terms of the right hand side of \eqref{PsiGlower1 in Phi}, 
   \begin{align*}
     & \Phi^{\{v_1\},\{v_2v_k\}}_{H} + \left(\Phi^{\{v_1v_{2}\},\{v_3v_k\}}_H - \Phi^{\{v_1v_3\},\{v_{2}v_k\}}_H\right) \\
     & = \Phi^{\{v_1v_3\},\{v_2v_k\}}_{H} + \Phi^{\{v_1\},\{v_2v_kv_3\}}_{H} + \left(\Phi^{\{v_1v_{2}\},\{v_3v_k\}}_H - \Phi^{\{v_1v_3\},\{v_{2}v_k\}}_H\right) \\
     & = \Phi^{\{v_1\},\{v_2v_kv_3\}}_{H} + \Phi^{\{v_1v_{2}\},\{v_3v_k\}}_H \\
     & = \Phi^{\{v_1\},\{v_kv_3\}}_{H}
   \end{align*}
   Thus the right hand side of \eqref{PsiGlower1 in Phi} equals
   \begin{equation}\label{rhs rearranged}
     \Phi^{\{v_1\},\{v_3v_k\}}_{H}  + \sum_{j=4}^{k-1} \left(\Phi^{\{v_1v_{j-1}\},\{v_jv_k\}}_H - \Phi^{\{v_1v_j\},\{v_{j-1}v_k\}}_H\right)+ \Phi^{\{v_k\},\{v_1v_{k-1}\}}_H
   \end{equation}
   which is the right hand side of \eqref{PsiGlower1 in Phi} for the lemma applied to a new graph $G'$ defined to be $G\backslash 2,3$ with a new edge $\ell$ joining vertices $v_1$ and $v_3$ along with the cycle $1,\ell,4,\ldots,k$.  Note that $H = G\backslash 1\cdots k = G'\backslash 1\ell 4 \cdots k$, and so inductively \eqref{rhs rearranged} is $\Phi^{\{v_1\},\{v_k\}}_{H}$.

   It remains to check the initial cases.
   $k=2$ is trivial.  Suppose $k=3$. Then, as desired,  
   \[
     \Phi^{\{v_1\},\{v_2v_3\}}_{H}+ \Phi^{\{v_3\},\{v_1v_2\}}_H = \Phi^{\{v_1\},\{v_3\}}_{H}
   \]

\end{proof}

\begin{prop} If $1,\ldots, k$ is a cycle in $G$ then
\begin{equation}\label{PsiGlower1} \Psi_1= \sum_{j=2}^k  \lambda_j x_j \Psi^{1,j}, \hbox{ where } \lambda_j =\pm 1
\end{equation}
\end{prop}
\begin{proof}
First note that by the contraction-deletion properties for Dodgson polynomials, any terms of \eqref{PsiGlower1} which do not contain $x_i$, for $2 \leq i \leq k$,  also appear
in \eqref{PsiGlower1} for the graph $G\q i$ and all such terms appear in this way.  Furthermore, they appear with the same signs since contracting an edge corresponds to setting the corresponding variable to zero in the Dodgson polynomials.  Clearly, contracting elements of a cycle gives a smaller cycle and so inductively it suffices to prove
the result holds just for the coefficient of $x_2\cdots x_k$.

Labelling the vertices as in Lemma \ref{help for PsiGlower1} and
translating into spanning forest polynomials
\begin{align*}
\Psi_1 & = \Phi_{G\backslash 1}^{\{v_1\},\{v_k\}} = x_2\cdots x_k \Phi_{G\backslash 1\cdots k}^{\{v_1\},\{v_k\}} + \text{ terms lower in $x_2,\ldots,x_k$} \\
x_2\Psi^{1,2} & = x_2\Phi_{G\backslash 1,2}^{\{v_1\},\{v_2v_k\}} = x_2\cdots x_k \Phi_{G\backslash 1\cdots k}^{\{v_1\},\{v_2v_k\}} + \text{ terms lower in $x_2,\ldots, x_k$}  \\
x_{k}\Psi^{1,k} & = x_k\Phi_{G\backslash 1,k}^{\{v_k\},\{v_1v_{k-1}\}} = x_2\cdots x_k \Phi_{G\backslash 1\cdots k}^{\{v_k\},\{v_1v_{k-1}\}} + \text{ terms lower in $x_2,\ldots, x_k$}
\end{align*}
and for $3 \leq j \leq k-1$
\begin{align*}
x_{j}\Psi^{1,j} & = x_j\left(\Phi_{G\backslash 1,j}^{\{v_1v_{j-1}\},\{v_jv_k\}} - \Phi_{G\backslash 1,j}^{\{v_1v_{j}\},\{v_{j-1}v_k\}}\right) \\
& = x_2\cdots x_k \left(\Phi_{G\backslash 1\cdots k}^{\{v_1v_{j-1}\},\{v_jv_k\}} -\Phi_{G\backslash 1\cdots k}^{\{v_1v_{j}\},\{v_{j-1}v_k\}}\right)  + \text{ terms lower in $x_2,\ldots, x_k$}
\end{align*}
By choosing the $\lambda_j$ appropriately, the result now follows from lemma \ref{help for PsiGlower1}.
\end{proof}

\begin{remark} Equation $(\ref{PsiGlower1})$ is  essentially dual to  lemma 31 in \cite{BrFeyn},  which states that for a graph $H$ in which edges $1,\ldots, k$ form a corolla (i.e. the set of edges
which meet a vertex), then 
$$\Psi^1_H = \sum_{j=2}^k \lambda_j  \Psi_H^{1,j}\hbox{ where } \lambda_j = \pm 1\ .$$
The proof  uses the Jacobi determinental formula (lemma 28 of \cite{BrFeyn}), and is easily seen to hold  for cographic matroids also (the graph matrix
 defined in \S2.2 of \cite{BrFeyn} generalizes to regular matroids by replacing the incidence matrix with the representation matrix of the matroid).
If $G$ denotes the graph in the statement of the proposition, and $H$ is  the dual matroid, then  the graph polynomials are related by
$\Psi_H(x_e) = \Psi_G(x_e^{-1}) \prod_{e\in E(G)} x_e $.
\end{remark}

\begin{cor} \label{corjacdual} Let $G$ be a graph with edge-connectivity\footnote{The edge-connectivity is the minimum number of edge cuts that splits the graph.} $\geq2$.
Let $I$ be the ideal in $\QQ[x_e,e\in E(G)\backslash \{1\}]$ spanned by $ \Psi^{1},\Psi^{12},\ldots,  \Psi^{1 k} $. Then $\Psi_1 \in \sqrt{I}.$
\end{cor}
\begin{proof} It follows from the Dodgson identity that
$$(\Psi^{1,j})^2 = [\Psi_j,\Psi^j]_1 =[\Psi,\Psi^j]_1 =\Psi^1\Psi_1^{j}- \Psi_1 \Psi^{1j}\in I \ ,$$
and so $\Psi^{1,j}\in \sqrt{I}$ for all $j\in E(G)$.
Since $G$ has edge-connectivity $\geq2$ it has a cycle containing edge $1$. Then equation  $(\ref{PsiGlower1})$ implies the result.
\end{proof}
\begin{cor} For any edge $e$ of $G$ as above, $X_{G\backslash e}\backslash (X_{G\backslash e}\cap X_{G\q e})$ is smooth.
\end{cor}

\subsection{Elimination of a variable}
The following lemma is a straightforward consequence of inclusion-exclusion.
\begin{lem}
Let $f_i$, $i\in I$, $h$, $g_j$, $j\in J$ be polynomials with index sets $I$ and $J$. Then
\begin{equation}\label{inex}
[f_i,hg_j]=[f_i,h]+[f_i,g_j]-[f_i,h,g_j]
\end{equation}
where $i$ and $j$ run through $I$ and $J$, respectively.
\end{lem}
\begin{proof}
Let $V(h)$ be the zero locus of $h$. Intersection with $V(h)$ gives $[f_i,hg_j,h]=[f_i,h]$. On the open complement $U$ of $V(h)$ we have
$$[V(f_i,hg_j)\cap U]=[V(f_i,g_j)\cap U]=[f_i,g_j]-[f_i,h,g_j].$$
Together we obtain (\ref{inex}).
\end{proof}

The next identity expresses the simultaneous elimination of a variable from the class of an ideal in the Grothendieck ring whose generators are all linear in that variable.
It generalizes lemma 3.3 in \cite{Stem} (or lemma 16 in \cite{BS}) to more than two generators.
\begin{prop}\label{propsimulred}
Let $f_1,\ldots, f_n$ denote polynomials which are linear in a variable $x$, and write $f_i=f^x_i x+ f_{ix}$ for $1\leq i\leq n$. Then $(\sum_1^{-1}=\sum_1^0=0)$
\begin{eqnarray}\label{15}
[f_1,\ldots, f_n]&=& [f^x_1,f_{1x},\ldots, f^x_n,f_{nx}] \, \Lef\nonumber\\
&&+\,[[f_1,f_2]_x,\ldots, [f_1,f_n]_x] -[f_1^x,\ldots, f_n^x]  \nonumber  \\
 &&+\,\sum_{k=1}^{n-2} \big( [f_1^x,f_{1x},\ldots, f_k^x,f_{kx}, [f_{k+1},f_{k+2}]_x,\ldots, [f_{k+1},f_n]_x] \\
 && \qquad\qquad\qquad  -\,[f^x_1,f_{1x},\ldots, f^x_k,f_{kx}]  \big) \ .\nonumber 
\end{eqnarray}
\end{prop}
\begin{proof}
We prove by induction a slightly generalized version of (\ref{15}) where we add a set of $x$-independent polynomials $g=g_1,\ldots,g_m$ to all ideals.
We consider $[X]=[g,f_1^xx+f_{1x},f_2,\ldots,f_n]$ with ambient space $\A^N$. On the zero locus $V(f_1^x)\subset\A^N$ of $f_1^x$ we have
$$[X,f_1^x]=[g,f_1^x,f_{1x},f_2,\ldots,f_n].$$
Let $U$ denote the open complement of $V(f_1^x)$ in $\A^N$. On $U$ the projection
$$V(X)\rightarrow  V(g,[f_1,f_2]_x,\ldots,[f_1,f_n]_x)\subset\A^{N-1}$$
is one to one. We hence have
$$[V(X)\cap U]=[g,[f_1,f_2]_x,\ldots,[f_1,f_n]_x]-[g,f_1^x,[f_1,f_2]_x,\ldots,[f_1,f_n]_x].$$
By the definition of the resultant we have
$$[g,f_1^x,[f_1,f_2]_x,\ldots,[f_1,f_n]_x]=[g,f_1^x,f_{1x}f_2^x,\ldots,f_{1x}f_n^x].$$
Equation (\ref{inex}) gives for the right hand side
$$[g,f_1^x,f_{1x}]+[g,f_1^x,f_2^x,\ldots,f_n^x]-[g,f_1^x,f_{1x},f_2^x,\ldots,f_n^x].$$
Putting these identities together we arrive at the formula
\begin{eqnarray}\label{16}
[X]&=&[g,f_1^x,f_{1x},f_2,\ldots,f_n]+[g,[f_1,f_2]_x,\ldots,[f_1,f_n]_x]\nonumber\\
&&-\,[g,f_1^x,f_{1x}]-[g,f_1^x,f_2^x,\ldots,f_n^x]+[g,f_1^x,f_{1x},f_2^x,\ldots,f_n^x].
\end{eqnarray}
For $n=1$ this reduces to
$$[g,f_1]=[g,f_1^x,f_{1x}]+[g]-[g,f_1^x].$$
The first term on the right hand side defines  a trivial $\A^1$ fibration  over $V(g,f_1^x,f_{1x})$ $\subset\A^{N-1}$.
Changing the ambient space for the first term to $\A^{N-1}$ we get a factor of $\Lef$ and the above equation establishes the initial case $n=1$.
To complete the induction over $n$ we can assume that the hypothesis holds for the first term on the right hand side of (\ref{16}) with $x$-independent polynomials
$g,f_1^x,f_{1x}$, yielding $(\sum_2^0=\sum_2^1=0)$
\begin{eqnarray}\label{17}
[g,f_1^x,f_{1x},f_2,\ldots,f_n]&=&[g,f_1^x,f_{1x},f_2^x,f_{2x},\ldots,f_n^x,f_{nx}]\,\Lef\nonumber\\
&&\hspace{-1.5cm}+\,[g,f_1^x,f_{1x},[f_2,f_3]_x,\ldots,[f_2,f_n]_x]-[g,f_1^x,f_{1x},f_2^x,\ldots,f_n^x]\nonumber\\
&&\hspace{-1.5cm}+\,\sum_{k=2}^{n-2}\big([g,f_1^x,f_{1x},\ldots,f_k^x,f_{kx},[f_{k+1},f_{k+2}]_x,\ldots,[f_{k+1},f_n]_x]\nonumber\\
&&-\,[g,f_1^x,f_{1x},\ldots,f_k^x,f_{kx}]\big).
\end{eqnarray}
The third term on the right hand side of (\ref{17}) cancels the last term on the right hand side of (\ref{16}) whereas
the second term on the right hand side of (\ref{17}) joins with the third term on the right hand side of (\ref{16}) to form the $k=1$ term in the
sum of (\ref{17}). Together with the remaining terms this completes the induction.
\end{proof}
Note that the left hand side of (\ref{15}) is symmetric under changing the order of the polynomials $f_i$ whereas the individual terms on the right hand side are not.
\subsection{Proof of theorem \ref{thmsing}} 

\begin{lem} \label{lemsingform} Let $G$ have edge-connectivity $\geq2$ with edges numbered $1,\ldots,N_G$. Then
\begin{equation}\label{id1}
[\Sing (X_G)] +[\Sing (X_{G\backslash 1})] = \Lef\, [\Psi^1,\Psi_1,\{ \Psi^{1n},\Psi^{n}_1\}_{n=2,\ldots,N_G}]  + [\Psi^1,\Psi_1]   \ .
\end{equation}
\end{lem}
\begin{proof}
The clas of the singular locus $\Sing(X_G)$ in affine space is given by $[\Psi,\Psi^1,\ldots,$ $\Psi^{N_G}]$. Apply 
proposition \ref{propsimulred} to the polynomials $\Psi,\Psi^1,\ldots, \Psi^{N_G}$, in order,  with respect to $x=x_1$. 
Each term in the sum is of the form:
$$[\Psi^1,\Psi_1, .\,., \Psi^{1k},\Psi^k_{1}, [\Psi^{k+1},\Psi^{k+2}]_1,.\,., [\Psi^{k+1},\Psi^{N_G}]_1] -  [\Psi^1,\Psi_1, .\,., \Psi^{1k},\Psi^k_{1}] \ .$$
By equation $(\ref{id3})$, each resultant $ [\Psi^{k+1},\Psi^{m}]_1$ is in the radical of the ideal spanned by $\Psi^1,\Psi_1$. Thus the reduced schemes
defined by these two ideals are the same and the total contribution is  zero in the Grothendieck ring.
It follows that all terms in the sum vanish, and we are left with only the first three terms:
\begin{equation}\label{singinproof1}
[\Sing(X_G)]=  [\Psi^1,\Psi_1, \Psi^{1i},\Psi^i_1]\, \Lef+[\Psi^1,[\Psi,\Psi^i]_1] -[\Psi^1,\Psi^{1i}]\ , 
\end{equation}
where in each expression, $i$ ranges from $2$ to $N_G$.  Clearly $[\Psi,\Psi^i]_1 = \Psi^1\Psi^{i}_1 - \Psi_1\Psi^{1i}$ and hence $[\Psi^1,[\Psi,\Psi^i]_1]=[\Psi^1,\Psi_1\Psi^{1i}]$.
Equation (\ref{inex}) gives
\begin{equation}\label{singinproof2}
[\Psi^1,\Psi_1\Psi^{1i}]=[\Psi^1,\Psi_1]+[\Psi^1,\Psi^{1i}]-[\Psi^1,\Psi_1,\Psi^{1i}].
\end{equation}
By corollary \ref{corjacdual}, we know that $\Psi_1 \in \sqrt{I}$, where $I$ is the ideal generated by $\Psi^1, \Psi^{1i}$.
Hence the right hand side of (\ref{singinproof2}) reduces to $[\Psi^1,\Psi_1]$ in the Grothendieck ring.
The third term on the right hand side of (\ref{singinproof1}) defines the singular locus of $\Psi^1$, which is the graph polynomial of $G\backslash 1$. 
\end{proof}

\begin{proof}[Proof of thm.\ \ref{thmsing}]
If $G$ is disconnected then $\Psi=0$ and the theorem holds true.

We now assume that $G$ is connected and prove the theorem by induction over $N_G$.

The initial case is the tree with 2 edges which has $\Psi=1$ and the theorem follows trivially.

Let $N_G\geq3$. If $G$ has edge-connectivity 1 then there exists an edge $e$ that cuts $G$. Hence $\Psi$ does not depend on $x_e$ and $\Sing(X_G)$ is a trivial line bundle implying
the statement of the theorem. We hence may assume that $G$ has edge-connectivity $\geq2$.

By reducing eq.\ $(\ref{id1})$ of lemma \ref{lemsingform} modulo $\Lef$, we get
$$[\Sing(X_G)] \equiv[\Psi^1,\Psi_1]   - [\Sing(X_{G\backslash 1})]     \mod \Lef\ .$$
By equation $(2)$ in the proof of Proposition-Definition 18 in \cite{BS}, we know that $h_G\leq N_G-2$ ($G$ has at least 3 vertices) implies
$$[\Psi^1,\Psi_1] = [\Psi_{G\backslash 1},\Psi_{G\q 1}] \equiv 0 \mod \Lef\ ,$$
from which we obtain that $[\Sing(X_G)] \equiv   - [\Sing(X_{G\backslash 1})] \mod \Lef.$ Because $G$ has edge-connectivity $\geq2$ we know that $G\backslash 1$ is connected with at least
three vertices. By induction $[\Sing(X_{G\backslash 1})]\equiv0 \mod \Lef.$
\end{proof} 

\section{graphs with subdivergences}

We show that a graph $G$ in $\phi^4$ theory which is not primitive (i.e. which contains a non-trivial divergent subgraph) has vanishing 
$c_2$ invariant. 
\subsection{Structure of a 3-edge join}

\begin{center}
\fcolorbox{white}{white}{
  \begin{picture}(321,93) (499,-58)
    \SetWidth{1.0}
    \SetColor{Black}
    \GOval(530,-17)(40,19)(0){0.882}
    \Line(541,15)(577,-17.2)
    \Line(548,-17)(577,-17)
    \Line(541,-48)(577,-16.8)
        \Vertex(576,-17){2}
    \GOval(643,-17)(40,19)(0){0.882}
    \Line(596,-17)(630.8,-47.8)
    \Line(596,-17)(624,-17)
    \Line(596,-17)(631.3,14.6)
         \Vertex(597,-17){2}
    \GOval(530,-17)(40,19)(0){0.882}
    \Text(496,-55)[lb]{\Large{\Black{$G_1$}}}
    \Text(558,4)[lb]{\small{\Black{$1$}}}
    \Text(558,-15)[lb]{\small{\Black{$2$}}}
    \Text(558,-30)[lb]{\small{\Black{$3$}}}
    \Text(576,-28)[lb]{\small{\Black{$v_1$}}}
    \Text(594,-28)[lb]{\small{\Black{$v_2$}}}
    \Text(615,4)[lb]{\small{\Black{$1$}}}
    \Text(615,-15)[lb]{\small{\Black{$2$}}}
    \Text(615,-31)[lb]{\small{\Black{$3$}}}
    \Text(605,-55)[lb]{\Large{\Black{$G_2$}}}
     \GOval(728,-17)(40,19)(0){0.882}
    \GOval(800,-17)(40,19)(0){0.882}
    \Line(743,8)(785,8)
    \Line(747,-17)(781,-17)
    \Line(743,-42)(785.3,-42)
    \Text(765,11)[lb]{\small{\Black{$1$}}}
    \Text(765,-13)[lb]{\small{\Black{$2$}}}
    \Text(765,-39)[lb]{\small{\Black{$3$}}}
  \end{picture}
}
\end{center}

\begin{defn}Let $G_1$, $G_2$ denote connected graphs with distinguished 3-valent vertices $v_1\in V(G_1)$, $v_2\in V(G_2)$. A \emph{three edge join}
of $G_1$ and $G_2$ is the graph  obtained by gluing $G_1\backslash v_1$ and $G_2 \backslash v_2$ along the 3 pairs of external half edges in some way.  Define $n$ edge joins similarly.
\end{defn}

Recall from \cite{BS} lemma 22 that the existence of a 3-valent vertex  in $G_1$ implies that
$$\Psi_{G_1} = f_0 (x_1x_2+x_1x_3+x_2 x_3) + (f_2+f_3)x_1 +(f_1+f_3)x_2 + (f_1+f_2)x_3 + f_{123}$$
where  the polynomials $f_i$ are defined by 
$$f_0 = \Psi_{G_1 \backslash \{1,2\} \q 3} \ , \ f_1 =\Psi^{2,3}_{G_1,1} \ , \ f_2 = \Psi^{1,3}_{G_1,2}  \ , f_3 = \Psi^{1,2}_{G_1,3} \  , \ f_{123}=\Psi_{G_1\q \{1,2,3\} }  \  ,$$ 
and satisfy the equation
\begin{equation} \label{fids}
f_0f_{123}= f_1f_2+f_1f_3+f_2f_3\ .
\end{equation}
Let $g_0,g_1,g_2,g_3,g_{123}$ denote the corresponding structure coefficients of the graph polynomial $\Psi_{G_2}$. 
The structure of a general 3-edge join is similar.
\begin{prop}  Let $G_1$, $G_2$ be as in the definition, and let $G$ be their 3-edge join, with the edges numbered accordingly. Then 
\begin{equation}\label{3EJquadric}
f_0 g_0 \Psi_{G} = A_1 A_2 + A_1 A_3 +A_2 A_3 \end{equation}
where
\begin{equation} 
A_i = f_0 g_0 x_i + f_i g_0 + f_0 g_i \qquad \hbox{ for } i=1,2,3
\end{equation}
\end{prop}
\begin{proof} It suffices to show that
$$
\Psi_{G} = f_0g_0 (x_1x_2+x_1x_3+x_2 x_3)  + \sum_{i=1}^3 (f_0g_{i+1}+f_0g_{i+2}+g_0f_{i+1}+g_0f_{i+2}) x_i$$
\begin{equation} \label{3EJformula}  + (f_{123}g_0+ f_0g_{123} + \sum_{i\neq j} f_i g_j)\end{equation}
where the indices in the second sum are taken modulo 3.

To prove \eqref{3EJformula} we recall the proof of Theorem 23 in \cite{BY} and apply the formula for a 3 vertex join with $G_1\backslash 123$ on one side and
$G\backslash (G_1\backslash 123)$
on the other side.  We get
\[
  \Psi_G = f_{123}g_0' + f_1g_2' + f_1g_3' + f_2g_1' + f_2g_3' + f_3g_1' + f_3g_2' + f_0g_{123}'
\]
where the $g'_w$ are the corresponding polynomials for the $G\backslash (G_1\backslash 123)$
side.  However, looking at each $g_w'$ in terms of the allowable spanning forests we see that
\begin{align*}
  g_0' & = g_0\\
  g_i' & = x_ig_0 + g_i  \text{ for $i\in\{1,2,3\}$}\\
  g_{123}' & = (x_1x_2+x_1x_3+x_2x_3)g_0 + x_1(g_2+g_3) + x_2(g_1+g_3) + x_3(g_1+g_2)\\
  &\quad +\, g_{123}
\end{align*}
which gives the desired result.
\end{proof} 

\subsection{The class of a 3-edge join in the Grothendieck ring}

The 3-edge join is simple enough that we can denominator reduce to zero and hence obtain that the $c_2$ invariant vanishes.
\begin{prop}\label{c2 3}
 Let $G$ be a 3-edge join of $G_1$, $G_2$ as defined above.  Let $4$ be an edge of $G_1\backslash 123$ and let $5$ be an edge of $G_2\backslash 123$.  Then 
\[
  {}^5\Psi_{G}(1,2,3,4,5) = 0.
\]
Consequently, the $c_2$ invariant of $G$ is $0\mod q$.
\end{prop} 

\begin{proof}
Consider $\Psi^{124, 135}_G$.  Monomials in this polynomial correspond to certain trees in $G\backslash 135 /24$.  Consequently, they correspond to certain spanning forests of $G\backslash 12345$ where the end points of $2$ and $4$ are coloured with three colours. 

Monomials of $\Psi^{124, 135}_G$ also correspond to trees in $G\backslash 124 / 35$, and hence to spanning forests of $G\backslash 12345$ where the end points of $3$ and $5$ are coloured with the same three colours.

But $4$ is in $G_1$ and $5$ is in $G_2$.  Thus the connected components of $G\backslash 123$ each contain at least two of the three colours.  Therefore, there is a colour which appears in both components.  But all vertices of the same colour must be in the same tree of the forest, so there can be no such spanning forest of $G\backslash 12345$.  Thus $\Psi^{124, 135}_G=0$.  The same argument holds with $4$ and $5$ swapped and so ${}^5\Psi_{G}(1,2,3,4,5) = \pm   [\Psi^{24,35}_1,  \Psi^{124,135}]=0$, by definition \ref{Def5inv}. 
\end{proof}

There is a direct way to show that the $c_2$ invariant of a 3 edge join vanishes.
\begin{prop}\label{c2 3 direct} Let $G$ be the 3-edge join of $G_1$, $G_2$ as defined above. Then 
$$[X_G] \equiv 0 \mod \Lef^3 $$
In other words, the $c_2$ invariant of $G$ vanishes in the Grothendieck ring.
\end{prop}
\begin{proof}
Let $U_{fg}$ and $U_{fg}'$ denote the open set $f_0, g_0 \neq 0$ in ambient space $\A^{N_G}$ and $\A^{N_G-3}$, respectively.
From $(\ref{3EJquadric})$, we have 
$$[X_G \cap U_{fg}] = \Lef^2 [U_{fg}']\ , $$
since the right-hand side of $(\ref{3EJquadric})$ defines a quadric in $\A^3$ whose class is $\Lef^2$.
Now let $U_f\subseteq\A^{N_G}$  denote  the   set $f_0\neq 0, g_0= 0$, and  likewise  let $U_g$  denote  the   set $g_0\neq 0, f_0= 0$.
From $(\ref{3EJformula})$, the polynomial $\Psi_G$ restricted to $U_f$ takes the form
\begin{equation}\label{genhyp}
(g_1+g_2) y_1 +(g_1+g_3) y_2+(g_2+g_3)y_3+ f_0g_{123}\end{equation}
where $y_i = f_0 x_i +f_i$ for $i=1,2,3$.  Consider the projection $ X_G \cap U_f  \rightarrow \A^{N_G-3}\cap U_f$.
By equation $(\ref{genhyp})$, the generic fiber   is a hyperplane in $\A^3$ whose class is $\Lef^2$. Otherwise, $g_1,g_2,g_3$ vanish and
there are two possibilities: if $g_{123}=0$ the fiber is isomorphic to  $\A^3$, otherwise it is empty.  We  therefore have
$$[X_G\cap U_f ] = M_{f_0}\times \big( \Lef^3 [g_0,g_1,g_2,g_3,g_{123}]  +\Lef^2([g_0] - [g_0,g_1,g_2,g_3])\big)\ ,$$
where $M_{f_0}= [\A^{N_{G_1}-3}\backslash V(f_0)]$ and all terms in brackets are viewed in $\A^{N_{G_2}-3}$. A similar equation holds for
$[X_G\cap U_g]$. Finally,
$$[X_G\cap V(f_0,g_0)]=\Lef^3 [f_0,g_0,\sum_{i\neq j} f_i g_j], $$
where the right hand side has ambient space $\A^{N_G-3}$. Writing
$[X_G]=[X_G\cap U_{fg} ] + [X_G\cap U_{f} ] +[X_G\cap U_{g} ] +[X_G\cap V(f_0,g_0) ]$ gives
\begin{eqnarray*}
[X_G] &= &    \Lef^{N_G-1}  - \Lef^2 ( M_{f_0} [g_0,g_1,g_2,g_3] + M_{g_0} [f_0,f_1,f_2,f_3]+[f_0][g_0])\\
 & +& \Lef^3 \Big( M_{f_0} [g_0,g_1,g_2,g_3,g_{123}] +  M_{g_0} [f_0,f_1,f_2,f_3,f_{123}]\\
 &&\quad +\,[f_0, g_0,\sum_{i\neq j} f_i g_j] \Big).
\end{eqnarray*}
In particular, the $c_2$ invariant of $G$ in the Grothendieck ring is
$$c_2(G)\equiv-\,M_{f_0} [g_0,g_1,g_2,g_3] - M_{g_0} [f_0,f_1,f_2,f_3]-[f_0][g_0] \mod \Lef\ .$$
However, it follows from \cite{BS} Proposition-Definition 18 (1) that $[f_0]\equiv 0 \mod \Lef$ since $f_0$ is  a graph polynomial and therefore  linear in every variable.
Thus $M_{f_0}\equiv 0 \mod \Lef$, and the same holds for $M_{g_0}$.  It follows that $c_2(G)\equiv0\mod\Lef$.  
\end{proof}

\subsection{Four edge joins}

The 4-edge joins are trickier, as we can take a denominator calculation part way there, and then must appeal to the Chevelley-Warning theorem via a separate argument.
\begin{lem} \label{lem4EJ}
Let $G$ be the 4-edge join of $G_1$ and $G_2$.  Let $5$ be an edge of $G_1\backslash 1234$ and let $6$ be an edge of $G_2\backslash 1234$.  Then 
\[
  D^6_G(1,2,3,4,5,6) = \pm(P_{1,\{ij,kl\}}P_{2,\{il,jk\}} + P_{1,\{il,jk\}}P_{2,\{ij,kl\}})
\]
for any $\{i,j,k,l\}=\{1,2,3,4\}$ where
    $P_{t,\{ij,kl\}}  = \pm \Psi_{G_t, i}^{jkl,kl(4+t)}\Psi_{G_t,k}^{ijl,ij(4+t)}$.
\end{lem}

\begin{proof} Recall that
  \[
  {}^5\Psi(2,3,4,5,6)
   = \Psi^{236,245}\Psi^{35,46}_2 - \Psi^{36,45}_2\Psi^{235,246}. 
  \]
  The graph $G\backslash 1$ is a 3-edge join, so by the proof of proposition \ref{c2 3}, and contraction-deletion, we immediately obtain
  $\Psi^{1236,1245}=\Psi^{1235,1246}=0$.  
  Thus we can denominator reduce edge $1$ to get
  \[
  D^6(1,2,3,4,5,6) = \Psi^{236,245}_1\Psi^{135,146}_2 - \Psi^{136,145}_2\Psi^{235,246}_1 \ .
  \]

  Now consider $\Psi^{236,245}_1$.  Let the vertex on the $G_i$ side of edge $1$ be $a_i$, and similarly for $b_i$, $c_i$, $d_i$ for edges $2$, $3$, $4$, and $5$ respectively.  Let $e_1$ and $f_1$ be the vertices of edge $5$ and $e_2$ and $f_2$ for edge $6$.  Then
  \begin{align*}
  & \Psi^{236,245}_1 \\
    & = \pm \Big(\Phi_{G\backslash 23456}^{\{a_1,a_2,e_1,e_2,d_1,c_2\},\{f_1,c_1\},\{f_2,d_2\}} 
    +\Phi_{G\backslash 23456}^{\{a_1,a_2,e_1,e_2,d_1\},\{f_1,c_1\},\{f_2,d_2,c_2\}} \\ 
    & \qquad + \Phi_{G\backslash 23456}^{\{a_1,a_2,e_1,e_2,c_2\},\{f_1,c_1,d_1\},\{f_2,d_2\}} 
    +\Phi_{G\backslash 23456}^{\{a_1,a_2,e_1,e_2\},\{f_1,c_1,d_1\},\{f_2,d_2,c_2\}} \\
    & \qquad \pm \text{the same four terms with $e_1$ and $f_1$ transposed, $e_2$ and $f_2$ transposed} \\
    & \qquad \quad \text{and both transposed with sign the sign of the permutation}\Big)
\end{align*}
The internal signs are consequences of corollaries 17 and 18 of \cite{BY}.  Focus on the first four terms. No edges with ends not in the partitions join the two halves of the graph and so
\begin{align*}
  & \Phi_{G\backslash 23456}^{\{a_1,a_2,e_1,e_2,d_1,c_2\},\{f_1,c_1\},\{f_2,d_2\}} 
    +\Phi_{G\backslash 23456}^{\{a_1,a_2,e_1,e_2,d_1\},\{f_1,c_1\},\{f_2,d_2,c_2\}} \\ 
    & \qquad + \Phi_{G\backslash 23456}^{\{a_1,a_2,e_1,e_2,c_2\},\{f_1,c_1,d_1\},\{f_2,d_2\}} 
    +\Phi_{G\backslash 23456}^{\{a_1,a_2,e_1,e_2\},\{f_1,c_1,d_1\},\{f_2,d_2,c_2\}} \\ 
    & = \Phi_{H_1}^{\{a_1,e_1,d_1\},\{f_1,c_1\}}\Phi_{H_2}^{\{a_2,e_2,c_2\},\{f_2,d_2\}}  
    + \Phi_{H_1}^{\{a_1,e_1,d_1\},\{f_1,c_1\}}\Phi_{H_2}^{\{a_2,e_2\},\{f_2,d_2,c_2\}}  \\
    & \qquad + \Phi_{H_1}^{\{a_1,e_1\},\{f_1,c_1,d_1\}}\Phi_{H_2}^{\{a_2,e_2,c_2\},\{f_2,d_2\}}  
    + \Phi_{H_1}^{\{a_1,e_1\},\{f_1,c_1,d_1\}}\Phi_{H_2}^{\{a_2,e_2\},\{f_2,d_2,c_2\}} \\
    & = \left(\Phi_{H_1}^{\{a_1,e_1,d_1\},\{f_1,c_1\}}+ \Phi_{H_1}^{\{a_1,e_1\},\{f_1,c_1,d_1\}}\right)\left(\Phi_{H_2}^{\{a_2,e_2,c_2\},\{f_2,d_2\}}+\Phi_{H_2}^{\{a_2,e_2\},\{f_2,d_2,c_2\}}\right) \\
    & = \Phi_{H_1}^{\{a_1,e_1\},\{f_1,c_1\}}\Phi_{H_2}^{\{a_2,e_2\},\{f_2,d_2\}}
\end{align*}
  where $H_t = G_t \backslash 1234$ for $t=1,2$.

  Calculating similarly on the remaining terms
  \begin{align*}
  & \Psi^{236,245}_1 \\
    & = \pm\left(\Phi_{H_1}^{\{a_1,e_1\},\{c_1,f_1\}}-\Phi_{H_1}^{\{a_1,f_1\},\{c_1,e_1\}}\right)\left(\Phi_{H_2}^{\{a_2,e_2\},\{d_2,f_2\}}-\Phi_{H_2}^{\{a_2,f_2\},\{d_2,e_2\}}\right)
  \end{align*}

  Let $A_t^{m,n} = \Phi_{H_t}^{\{m_t,e_t\},\{n_t,f_t\}}-\Phi_{H_t}^{\{m_t,f_t\},\{n_t,e_t\}}$ for $t \in \{1,2\}$ and  $m,n\in\{a_t,b_t,c_t,d_t\}$.  Note that $A_t^{m,n} = -A_t^{n,m}$. 
The preceding calculations show that 
  \[
  \Psi^{236,245}_1 = \pm A_1^{a,c}A_2^{a,d}.
  \]
  Calculating similarly we get
  \begin{align*}
    \Psi^{135,146}_2 & = \pm A_1^{b,d}A_2^{b,c} \\
    \Psi^{136,145}_2 & = \pm A_1^{b,c}A_2^{b,d} \\
    \Psi^{235,246}_1 & =\pm  A_1^{a,d}A_2^{a,c}
  \end{align*}

Furthermore, $A_1^{a,b} = \pm \Psi_{G_1,1}^{234,345} = \pm \Psi_{G_1,2}^{134,345}$ and similarly for the other $A_t^{m,n}$.  Thus $P_{t,\{12,34\}} = \pm A_t^{a,b}A_t^{c,d}$,
    $P_{t,\{13,24\}} = \pm A_t^{a,c}A_t^{b,d}$, and
    $P_{t,\{14,23\}}  = \pm A_t^{a,d}A_t^{b,c}$.
  Choosing the signs on the $P_{t,\{ij,kl\}}$ appropriately we get
  \[
    D^6_G(1,2,3,4,5,6) = \pm(P_{1,\{13,24\}}P_{2,\{14,23\}} + P_{1,\{14,23\}}P_{2,\{13,24\}}).
  \]

  The other permutations of $1,2,3,4$ in the expression for $D^6$ in the statement of the theorem must hold by symmetry.  We can also verify them directly from the identity
  \[
    A_t^{a,b}A_t^{c,d} - A_t^{a,c}A_t^{b,d} +  A_t^{a,d}A_t^{b,c} = 0
  \]
  for $t=1,2$ which can be checked by expanding each term in spanning forests.
\end{proof}

\begin{remark} The expression for $D^6$ in the previous lemma is symmetric under various twisting operations.  
{}From the expressions for the $P_{t,\{ij,kl\}}$ in terms of the $A_t^{m,n}$ we see that each $P_{t,\{ij,kl\}}$ is invariant under the permutations $(1234)$, $(2143)$, and $(4321)$.
If we denote the four external vertices of $G_i\backslash \{1,2,3,4\}$
 by $v^i_{1},\ldots, v^i_4$ (not necessarily distinct), then any 4-edge join of $G_1$ and $G_2$ is obtained by connecting $v^1_i$ to  $v^2_{\sigma(i)}$ for $i=1,\ldots, 4$, where $\sigma$ is any permutation of $1,2,3,4$. Denote the corresponding $4$-edge join by $G_1\cup_{\sigma} G_2$. 
Then we have the following twisting identities:
 \begin{equation}\label{sym eq}
   D^6_{G_1 \cup_{id} G_2}= \pm D^6_{G_1 \cup_{\sigma} G_2}\ , 
 \end{equation}
 for all $\sigma \in V=\{(1234), (2143), (3412), (4321)\} $. 
\end{remark} 

\begin{prop}\label{c2 4} Let $G$ be a 4-edge join of $G_1$, $G_2$, and  let $A_i = G_i \backslash \{1,2,3,4\}$. If 
$2h_G\leq N_G$  and $ 2 h_{A_2}\leq N_{A_2} -2$ then  $c_2(G)_q\equiv 0 \mod q$.
\end{prop}
\begin{proof} By theorem \ref{thmDenom}, the $c_2$ invariant of $G$ is computed by its denominator reduction.  By lemma \ref{lem4EJ}, the zero locus  of 
$D^6_G(1,2,3,4,5,6)$ is given by 
$Z=V(P_1Q_2+Q_1P_2)\subset \A^{N_{A_1}-1} \times \A^{N_{A_2}-1}$ for polynomials $P_i,Q_i$ defined over $\ZZ$.  Consider the projection 
$\pi_1: \A^{N_{A_1}-1} \times \A^{N_{A_2}-1}\rightarrow \A^{N_{A_1}-1} $ onto the $A_1$ coordinates (minus edge 5), and let $Z_1=\pi_1(Z).$ 
By contraction-deletion, one sees that  $\deg P_2= \deg Q_2 =2 h_{A_2} $, and so the fibers of $Z$ over $Z_1$ are of degree $2h_{A_2}$
in $\A^{N_{A_2}-1}$.   Let $q=p^n$ where $p$ is prime, and let $\overline{Z}, \overline{Z}_1$ denote the reductions mod $p$. 
Since $2h_{A_2} < N_{A_2} -1$, the Chevalley-Warning theorem implies that $[\overline{Z}\cap \pi_1^{-1}(x)]_q\equiv 0 \mod q$ for all $x\in \overline{Z}_1$. 
Therefore $[Z]_q = \sum_{x\in  \overline{Z}_1}ब[Z\cap \pi_1^{-1}(x)]_q\equiv 0 \mod q$. 
\end{proof}

\subsection{Vanishing of $c_2$ for non-primitive graphs}
\begin{thm} Let $G$ be a connected graph in $\phi^4$  which is overall log-divergent. If $G$ has a non-trivial divergent subgraph then 
$c_2(G)_q \equiv 0 \mod q$. 
\end{thm}

\begin{proof} Let $\gamma$ be a divergent subgraph of $G$. Since $\gamma\in \phi^4$,
 it has at most 4 external edges, and so $G$ can be written as a 2, 3, or 4-edge join.  In the case of a 2-edge join, $G$ is in particular 2-vertex reducible, so by 
 proposition 36 of \cite{BY}, it has weight drop. In the case of a 3-edge join, the statement follows from  proposition \ref{c2 3} or \ref{c2 3 direct}. In the case of a 4-edge join,
 apply proposition  \ref{c2 4} with $A_1=\gamma$ and $G_2=G/\gamma$. Since $2 h_G=N_G$ and $2 h_{A_1}\geq N_{A_1}$, we deduce that $2h_{A_2} \leq N_{A_2}-2$.
 In  all cases $c_2(G)_q\equiv 0 \mod q$.
\end{proof}

 \begin{remark}  If one knew the completion conjecture for $c_2$ invariants  \cite{BS}, then 
 in the previous theorem it would be enough to know that $c_2(G)$ vanishes for 2 and 3-edge joins only. 
 \end{remark} 

\subsection{Insertion of a subgraph}
If we strengthen the hypotheses in the cases of the 3 and 4-edge joins, then we can obtain stronger conclusions and also clarify what fails in the case of higher joins.

Let $G$ be an overall log-divergent $\phi^4$ graph.  Suppose $H$ is a subgraph of $G$ with $2m$ external edges.  Then $N_G=2h_G$ and $N_H = 2h_H -2 + m$.  In this case $G$ is a $k$-edge join of $G_1= G\q H$ and $G_2$ where $G_2$ is $H$ with those external edges of $H$ which became internal edges of $G$ all attached to an additional vertex.  In particular, $k \leq 2m$.  In the proposition below we will never require the valence restrictions of a $\phi^4$ graph, only the relation between the edges and cycles for $H$, and so we drop the superfluous restrictions.

\begin{prop}\label{reduce H}
Let $G$ be a $k$-edge join of $G_1$ and $G_2$, with the join edges labelled $1,\ldots, k$.  Let $H=G_2\backslash \{1,\ldots, k\}$ and let $m = N_H-2h_H + 2$.  
Suppose
  all edges of $H$ can be denominator reduced in $G$.
Let $P$ be the denominator after these reductions.  Suppose further that
  $P$ can be written in the form $\sum \pm \Phi_{G\backslash H}^R\Phi_{G\backslash H}^{R'}$ with only the vertices where $H$ is attached  involved in the partitions.
Then
\[
P = \begin{cases} 0 & \text{ if $m < 2$} \\
  \Psi_{G\q H}^2 & \text{ if $m=2$}\\
  \Psi_{G\q H}Q & \text{ if $m=3$}
  \end{cases}
\] 
for some $Q$.
\end{prop}

Before proving this result let us consider it briefly.  From the preceding discussion we see that if we are in $\phi^4$ and $H$ is a vertex subdivergence of $G$, then we have $m=2$ and $k=3$ or $4$ in the proposition.  Thus with the hypotheses of the proposition we conclude that $P = \Psi_{G\q H}^2$, and hence in this case we have another way to see that the $c_2$ invariant is zero.

The hypothesis on $P$ deserves further explanation.  If the denominator one step before $P$ was expressible as a product of two Dodgson polynomials then $P$ will be a difference of products of pairs of Dodgson polynomials, and since every Dodgson polynomial can be written as a signed sum of spanning forest polynomials, we get the desired hypothesis on $P$.  

The proof of the proposition is a degree counting exercise.
\begin{proof}
Any 5-invariant in $G$ has degree $2h_G-5$ and each subsequent denominator reduction  decreases the degree of the denominator by $1$, so 
\[
\deg P = 2h_G - N_H = 2(h_{G\q H}+h_H) - N_H = 2h_{G\q H} + 2 - m
\] 
$\Psi_{G\backslash H}$ has degree $h_{G\q H}-k+1$.  Thus a spanning forest of $G\backslash H$ with $i$ trees has degree $h_{G\q H}-k+i$.
A partition involving only the vertices where $H$ is attached has at most $k$ parts.
Thus the maximum degree of a spanning forest polynomial associated to such a partition is $h_{G\q H}$.  

If $m<2$ then $P$ has degree at least $2h_{G\q H}+1$, but the maximum degree of a product of two spanning forest polynomials of the desired form is $2h_{G\q H}$, so $P=0$.

If $m=2$ then $P$ has degree $2h_{G\q H}$.  Thus $P$ is a sum of product of pairs of spanning forest polynomials each with $k$ trees.  But there is only one spanning forest polynomial with $k$ trees and $k$ vertices in the partition: each vertex is in a different part.  Furthermore this spanning forest polynomial is the same as the spanning forest polynomial with one part when all $k$ vertices are identified.  But $G\backslash H$ with the vertices where $H$ is connected identified is exactly $G\q H$.  Thus $P=\Psi_{G\q H}^2$.

If $m=3$ then $P$ has degree $2h_{G\q H}-1$.  This means that each term of $P$ is a product of a spanning forest polynomial with $k$ trees and one with $k-1$ trees.  But as shown in the previous paragraph the only spanning forest of the desired form with $k$ trees is $\Psi_{G\q H}$.  Thus we can factor out $\Psi_{G\q H}$ and we obtain $P = \Psi_{G\q H}Q$ where $Q$ is a linear combination of spanning forest polynomials with $k-1$ trees.
\end{proof}

Something similar happens if we reduce the outer graph rather than the inserted graph.  For insertions of $\phi^4$ primitive graphs into primitive graphs this would be the case $m=3$ and $k=3$ or $m=4$ and $k=4$.
\begin{cor}
Using the notation of proposition \ref{reduce H}, assume we can additionally reduce the $k$ edges of the join.  Let $\widetilde{P}$ be the resulting denominator and assume $\widetilde{P}$ satisfies the property satisfied by $P$ in proposition \ref{reduce H}.  Then
$$
  \widetilde{P} = \begin{cases}
    \Psi_{G\backslash G_2}^2 & \text{ if $m=k$}\\
    \Psi_{G\backslash G_2}\widetilde{Q} & \text{ if $m = k-1$} \end{cases}
$$
for some $Q$.
\end{cor}

\begin{proof}
  Begin as in the proof of proposition \ref{reduce H}.  Then 
  \[
  \deg \widetilde{P} = \deg P - k = 2h_{G\q H} + 2 - m -k.
  \]
  $\Psi_{G\backslash H}$ has degree $h_{G\q H}-k+1$.  This is the unique spanning forest polynomial of $G \backslash H$ of this degree and no such spanning forest polynomial can have smaller degree.  The result follows.
\end{proof}

\section{Denominator identities and $c_2$} 

Given that denominator reduction computes the $c_2$ invariant
it is natural to ask how the $c_2$ invariant relates to identities between denominators.  The double triangle identity \cite{BY} is also an identity of $c_2$ invariants, and is a major tool to predict the weight of Feynman graphs.  For denominator identities with more than two terms the situation is more subtle.  Two important such identities are the STU-type identity coming from splitting a 4-valent vertex, and the 4 term relation.

\subsection{An identity for 4-valent vertices} 
Let $G$ be a graph containing a 4-valent vertex with vertices $1,2,3,4$ pictured below on the left, where the white vertices denote vertices which are connected to the rest of the graph.  Resolve the 4-valent vertex  into three smaller graphs as shown:

\begin{center}
\fcolorbox{white}{white}{
  \begin{picture}(362,72) (432,-60)
    \SetWidth{1.0}
    \SetColor{Black}
    \COval(634,-12)(2,2)(0){Black}{White}
    \COval(664,18)(2,2)(0){Black}{White}
    \COval(664,-42)(2,2)(0){Black}{White}
    \COval(693,-12)(2,2)(0){Black}{White}
    \COval(553,-12)(2,2)(0){Black}{White}
    \COval(582,18)(2,2)(0){Black}{White}
    \COval(582,-42)(2,2)(0){Black}{White}
    \COval(612,-12)(2,2)(0){Black}{White}
    \Line(553,-12)(583,18)
    \Line(582,-42)(612,-12)
    \COval(715,-12)(2,2)(0){Black}{White}
    \COval(745,18)(2,2)(0){Black}{White}
    \COval(745,-42)(2,2)(0){Black}{White}
    \COval(775,-12)(2,2)(0){Black}{White}
    \COval(553,-12)(2,2)(0){Black}{White}
    \COval(582,18)(2,2)(0){Black}{White}
    \COval(582,-42)(2,2)(0){Black}{White}
    \COval(612,-12)(2,2)(0){Black}{White}
    \Line(634,-12)(693,-12)
    \Line(664,18)(664,-41)
    \Line(745,18)(775,-12)
    \Line(715,-12)(745,-42)
    \Vertex(465,-12){2}
    \COval(435,-12)(2,2)(0){Black}{White}
    \COval(465,18)(2,2)(0){Black}{White}
    \Line(465,18)(465,-41)
    \Line(435,-12)(494,-12)
    \COval(465,-42)(2,2)(0){Black}{White}
    \COval(494,-12)(2,2)(0){Black}{White}
    \Text(459,2)[lb]{\small{\Black{$1$}}}
    \Text(479,-10)[lb]{\small{\Black{$2$}}}
    \Text(470,-32)[lb]{\small{\Black{$3$}}}
    \Text(450,-10)[lb]{\small{\Black{$4$}}}
      \Text(459,20)[lb]{\tiny{\Black{$1$}}}
    \Text(499,-10)[lb]{\tiny{\Black{$2$}}}
    \Text(470,-48)[lb]{\tiny{\Black{$3$}}}
    \Text(430,-10)[lb]{\tiny{\Black{$4$}}}
    \Text(435,-55)[lb]{{\Black{$G$}}}
    \Text(562,4)[lb]{{\Black{$e$}}}
    \Text(592,-28)[lb]{{\Black{$f$}}}
    \Text(538,-58)[lb]{{\Black{$G_{\{14,23\}}$}}}
 \Text(620,-58)[lb]{{\Black{$G_{\{13,24\}}$}}}
 \Text(700,-58)[lb]{{\Black{$G_{\{12,34\}}$}}}
    \Text(671,2)[lb]{{\Black{$e$}}}
    \Text(648,-24)[lb]{{\Black{$f$}}}
    \Text(764,6)[lb]{{\Black{$e$}}}
    \Text(735,-26)[lb]{{\Black{$f$}}}
  \end{picture}
}
\end{center}

There exist three spanning forest polynomials $A,B,C$ such that
$$\Psi_{G_{\{14,23\}}^{e,f}}=\pm( A-B) \ , \qquad  \Psi_{G_{\{13,24\}}^{e,f}}=\pm(A-C) \ ,   \qquad  \Psi_{G_{\{12,34\}}^{e,f}}=\pm(B-C), $$
where $A,B,C\in \ZZ[x_5,\ldots, x_{N_G}]$ (by \cite{BY}, Example 13). Specifically, 
$$A= \Phi_G^{\{1,2\}, \{3,4\} }  \ , \  B= \Phi_G^{\{1,3\}, \{2,4\} }  \ ,\  C= \Phi_G^{\{1,4\}, \{2,3\} } \ . $$
\begin{lem} Consider a fifth edge $5$ in $G$. Then
$$\pm^5\Psi_G(1,2,3,4,5) =[A,B ]_{x_5} +[B,C]_{x_5} + [C,A]_{x_5} $$
\end{lem}
\begin{proof} For any partition $p$ of $\{1,2,3,4\}$ into two sets, 
$\Psi_G^{p} = \pm \Psi_{G_{p}}^{i,j}$. One of the many definitions of the five-invariant is:
$$\pm^5\Psi_G(1,2,3,4,5)=\pm[\Psi_G^{12,34}, \Psi_G^{13,24}]_{x_5}=\pm[B-C, A-C]_{x_5}\ .$$
The result follows by linearity of the resultant.
\end{proof}

This is not a typical denominator identity since it uses the decomposition into $A$, $B$, and $C$.  It becomes a true denominator identity when edge $5$ forms a triangle with $1$ and $2$.  In this case $G_{\{12,34\}}$ has a double edge which gives a denominator of $0$ when those edges are reduced and so only two terms remain on the right hand side.  
Specifically, we get the following proposition.
\begin{prop}\label{split4t}Let $G$ be as illustrated above and let edge $5$ form a triangle with edges $1$ and $2$.
  Choose any $6$th edge from $G$, then
  \[
    D^6_G(1,2,3,4,5,6) =  \pm D^4_{G_{\{14,23\}}}(5,e,f,6)  \pm D^4_{G_{\{13,24\}}}(5,e,f,6)
  \]
where 
\[
  D_G^4(i,j,k,l) = \pm  \Psi^{ij,kl}_G\Psi^{ik,jl}_G
\]
which depends on the order of the arguments.
\end{prop}

\begin{proof} Let $A,B,C$ be defined as above. In the quotient $G \q 5$, the vertices $1$ and $2$ are identified, which implies that 
$B$ and $C$ vanish at $x_5=0$, by contraction-deletion. By the previous lemma
  \begin{eqnarray}
   \pm {}^5\Psi_G(1,2,3,4,5)  &= & [A,B]_{x_5} +  [B,C]_{x_5}+ [C,A]_{x_5} \nonumber \\
     &= & (C^5 - B^5)A_5 \nonumber 
  \end{eqnarray}
  where we write $A=A^5 x_5+ A_5$, and so on, as usual. Then, using the fact that $B_5=C_5=0$, we deduce that
  \begin{eqnarray}
   \pm {}^6\Psi_G(1,2,3,4,5,6)  &= & [A_5,C^5]_{x_6} -  [A_5,B^5]_{x_6} \nonumber \\
     &= &  [ A_5-C_5, A^5-C^5]_{x_6} -[A_5-B_5, A^5-B^5 ]_{x_6}  \nonumber 
  \end{eqnarray}
  By the Dodgson identity, it is true for any graph polynomial $\Psi$ that 
  $$[\Psi^{i,j}_k, \Psi^{ik,jk}]_{x_{l}} = \Psi^{il,jk} \Psi^{ik,jl}\ .$$
  Writing $A-C= \Psi_{G_{\{13,24\}}}^{i,j}$ and $A-B= \Psi_{G_{\{14,23\}}}^{i,j}$, we obtain the statement of the proposition.
  \end{proof}

If we fix the signs in the $D_G^4$ by defining $D_G^4 = \Psi^{ij,kl}_G\Psi^{ik,jl}_G$ then, following the signs through the above proof, we obtain
\[
D^6_G(1,2,3,4,5,6) =  \pm \left(D^4_{G_{\{14,23\}}}(5,e,f,6) - D^4_{G_{\{13,24\}}}(5,e,f,6)\right)
\]

\begin{remark}
The preceding proposition implies the double triangle identity from \cite{BY}.
It also explains the ad hoc identities from subsection 4.6 of \cite{BY} if one also keeps track of the signs from the proof of the proposition.
For example, using the notation of that paper, if we apply proposition \ref{split4t} to the middle left vertex of $8_a$ then we obtain (with signs) $ 6_2-   6_3$ giving the polynomial $(xy+yz+xz)-xz = y(x+z)$. Applying the proposition to the top left vertex of $8_b$ we obtain two permutations of $6_3$ giving the polynomial $yz + xy$.  

Likewise, applying proposition \ref{split4t} twice to $10_b$ gives the three different permutations of $6_3$ and so correctly computes $\rho(10_b)$.  Consequently, these types of identities are no longer ad hoc, but come from splitting 4-valent vertices.  
\end{remark}

\subsection{4-term relation}

One very important relation in mathematics \cite{BNvass}, which is also found in quantum field theory \cite{bk4tr}, is the 4-term relation.  The $c_2$ invariant does not satisfy this relation, but it is nonetheless a true identity of denominators.

Let
\[
G_1 = \raisebox{-1.6cm}{\includegraphics{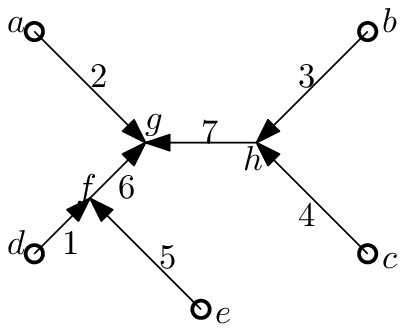}} \quad 
G_2 = \raisebox{-1.6cm}{\includegraphics{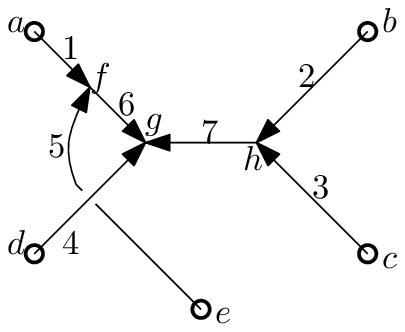}}
\]
\[
G_3 = \raisebox{-1.6cm}{\includegraphics{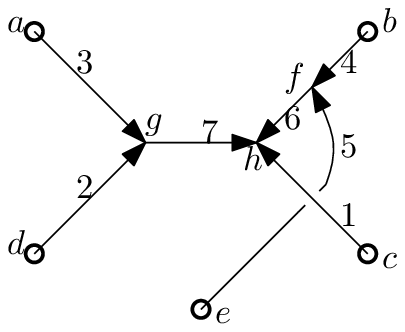}} \quad 
G_4 = \raisebox{-1.6cm}{\includegraphics{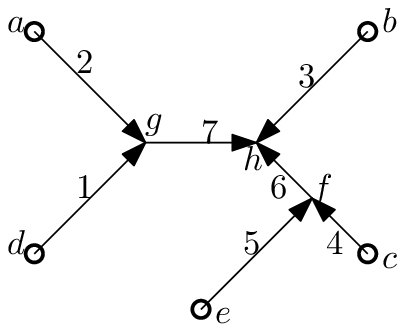}}
\]
each with the same external graph attached to the white vertices. 

\begin{thm}\label{thm 4tr}
\[
\pm   D^7_{G_1}  \pm  D^7_{G_2}  \pm D^7_{G_3}  \pm D^7_{G_4} = 0
\]
\end{thm}

\begin{proof}
  Beginning with $G_1$ calculate
  \[
\pm  {}^5\Psi_{G_1}(1,2,3,4,5) = \Psi^{12,34}_{G_1,5}\Psi^{145,235}_{G_1} - \Psi^{1
4,23}_{G_1,5}\Psi^{125,345}_{G_1}
  \]
  Since only edges $1$, $5$, and $6$ are adjacent to vertex $f$, reducing by edge $6$ gives
  \[
 \pm D^6_{G_1}(1,2,3,4,5,6) = \Psi^{126,346}_{G_1,5}\Psi^{145,235}_{G_1,6} - \Psi^{146,236}_{G_1,5}\Psi^{125,345}_{G_1,6}
  \]
  Since only edges $2$, $6$, and $7$ are adjacent to vertex $g$, reducing by edge $7$ gives
  \begin{align*}
 \pm D^7_{G_1}(1,2,3,4,5,6,7) & = \Psi^{126,346}_{G_1,57}\Psi^{1457,2357}_{G_1,6} - \Psi^{146,236}_{G_1,57}\Psi^{1257,3457}_{G_1,6} \\
  & = \Psi^{126,346}_{G_1,57}\Psi^{1457,2357}_{G_1,6}
  \end{align*}
  since $\Psi^{1257,3457}_{G_1,6}=0$ by the vanishing property \S \ref{subsectDodg} (4)  as $h$ is 3-valent.
 
  Let
  \[
    H = \raisebox{-1.6cm}{\includegraphics{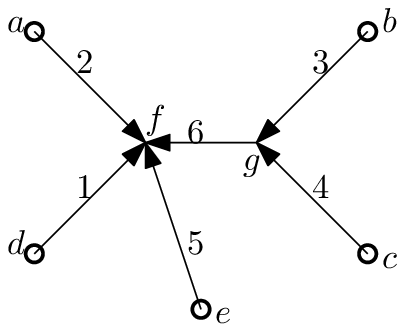}}
  \]
  Note that $\Psi^{1457,2357}_{G_1,6} = \pm \Phi^{\{a,b,c,d,e\}}_{H \backslash 1,2,3,4,5,6}$ since both correspond to the same spanning forest polynomials.  
  
  Furthermore, $\Psi^{126,346}_{G_1,57} = \pm \Psi^{156,234}_H$ since both correspond  up to a sign to \[\Phi^{\{b,d\},\{c,e\}}_{H\backslash 1,2,3,4,5,6} - \Phi^{\{b,e\},\{c,d\}}_{H\backslash 1,2,3,4,5,6}.\]  
  Thus, up to signs,
  \[
    D^7_{G_1}(1,2,5,3,4,6,7) = \Psi^{156,234}_H \Phi^{\{a,b,c,d,e\}}_{H \backslash 1,2,3,4,5,6}
  \]
  Arguing similarly for $G_2$, $G_3$, and $G_4$ we get
  \begin{align*}
    D^7_{G_2}(1,2,3,4,5,6,7) & = \pm \Psi^{146,236}_{G_2,57}\Psi^{1257,3457}_{G_2,6}  = \pm \Psi^{134,256}_{H}\Phi^{\{a,b,c,d,e\}}_{H \backslash 1,2,3,4,5,6}\\
    D^7_{G_3}(1,2,3,4,5,6,7) & = \pm \Psi^{146,236}_{G_3,57}\Psi^{1257,3457}_{G_3,6}  = \pm \Psi^{124,356}_{H}\Phi^{\{a,b,c,d,e\}}_{H \backslash 1,2,3,4,5,6} \\
    D^7_{G_4}(1,2,3,4,5,6,7) & = \pm  \Psi^{126,346}_{G_4,57}\Psi^{1457,2357}_{G_4,6}  = \pm \Psi^{123,456}_{H}\Phi^{\{a,b,c,d,e\}}_{H \backslash 1,2,3,4,5,6}
  \end{align*}
  All together
  \begin{align*}
    & \pm  D^7_{G_1} \pm D^7_{G_2} \pm  D^7_{G_3} \pm D^7_{G_4} \\
    & = (\pm \Psi^{234,156}_H\pm \Psi^{134,256}_{H}\pm \Psi^{124,356}_{H}\pm \Psi^{123,456}_{H})\Phi^{\{a,b,c,d,e\}}_{H \backslash 1,2,3,4,5,6} = 0
  \end{align*}
which vanishes for appropriate sign choices   by the Pl\"ucker identity with $n=3$ (\S\ref{subsectDodg}).
\end{proof}

Although there is no well-defined way to fix the signs in the denominator reduction for a single graph in general, we can determine signs for the $D^7$'s of the above graphs.
For this, viewing the arguments to the 5 invariant as ordered, we can choose the sign given by the positive sign in the expression in definition \ref{Def5inv}, following that, fix the signs of later reductions where either the constant or quadratic term vanishes by choosing the sign of the constant term in the previous step.  With these conventions and the order and orientations given in the illustrations, the above proof gives that 
\begin{equation} 
  D^7_{G_1}  - D^7_{G_2}  + D^7_{G_3}  - D^7_{G_4} = 0
\end{equation}

This identity strongly suggests a connection to the 4-term relation for chord diagrams in knot theory \cite{BNvass}.  The other key identity in chord diagrams is the one-term relation.  For denominators, the one-term relation is the fact that graphs of the form
\[
\includegraphics{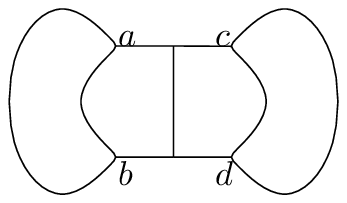}
\]
are zero after integrating the five indicated edges.  To see this, write the 5-invariant of the five edges joining $a$, $b$, $c$, and $d$ in spanning forest polynomials,
\[
 \pm \left(\Phi^{\{a,c\},\{b,d\}} - \Phi^{\{a,d\},\{b,c\}}\right)\Phi^{\{a,b,c,d\}}
\]
This is zero since in both terms of the first factor there are parts which appear in both components of the graph and no edges remaining to join them.

Note that, unfortunately the four-term relation does not hold at the level of the $c_2$ invariant.  As an example take $P_{7,11}$ from \cite{CENSUS}.  From the illustration in that paper, label the vertices counterclockwise from 3 o'clock starting with label $0$.  Next make a double triangle expansion of vertex $1$ in triangle $012$ so that the new vertex is adjacent to vertex 6. Remove the new vertex.  This graph has the same $c_2$ invariant as $P_{7,11}$.  However if we use the seven edges $03$, $08$, $01$, $12$, $25$, $14$, and $27$ with vertex $7$ playing the role of $e$, then the four $c_2$ invariants do not cancel.

\bibliographystyle{plain}
\renewcommand\refname{References}

\end{document}